\documentclass[12pt]{amsart}
\usepackage{a4wide,enumerate,mathrsfs,color}
\allowdisplaybreaks

\let\pa\partial
\let\na\nabla
\let\eps\varepsilon
\newcommand{\N}{{\mathbb N}}
\newcommand{\R}{{\mathbb R}}
\newcommand{\diver}{\operatorname{div}}

\newtheorem{theorem}{Theorem}
\newtheorem{lemma}[theorem]{Lemma}

\newtheorem{corollary}[theorem]{Corollary}

\begin{document}

\title[An incompressible Navier-Stokes-Maxwell-Stefan system]{Analysis of an
incompressible Navier-Stokes-Maxwell-Stefan system}

\author[X. Chen]{Xiuqing Chen}
\address{School of Sciences, Beijing University of Posts and Telecommunications,
  Beijing 100876, China}
\email{buptxchen@yahoo.com}

\author[A. J\"ungel]{Ansgar J\"ungel}
\address{Institute for Analysis and Scientific Computing, Vienna University of
	Technology, Wiedner Hauptstra\ss e 8--10, 1040 Wien, Austria}
\email{juengel@tuwien.ac.at}

\date{\today}

\thanks{The first author acknowledges support from the National Science
Foundation of China, grant 11101049. The second author was partially supported
by the Austrian Science Fund (FWF), grants P22108, P24304, I395, and W1245, and
the Austrian-French Project of the Austrian Exchange Service (\"OAD)}

\begin{abstract}
The incompressible Navier-Stokes equations coupled to the Maxwell-Stefan relations
for the molar fluxes are analyzed in bounded domains with no-flux
boundary conditions. The system models the dynamics of a multicomponent
gaseous mixture under isothermal conditions.
The global-in-time existence of bounded weak
solutions to the strongly coupled model and their exponential decay to the homogeneous
steady state are proved. The mathematical difficulties are due to the singular
Maxwell-Stefan diffusion matrix, the cross-diffusion terms, and the
Navier-Stokes coupling. The key idea of the proof is the use of a new entropy
functional and entropy variables, which allows for a proof of positive lower and
upper bounds of the mass densities without the use of a maximum principle.
\end{abstract}



\maketitle


\section{Introduction}

The dynamics of a multicomponent gaseous mixture can be described by
the Navier-Stokes equations, which represent the balance of mass, momentum,
and energy, and the Maxwell-Stefan equations, which model the diffusive
transport of the components of the mixture. Applications arise, for instance,
from physics (sedimentation, astrophysics), medicine (dialysis, respiratory airways),
and chemistry (electrolysis, ion exchange, chemical reactors) \cite{WeKr00}.
The understanding of the analytical structure of coupled
Navier-Stokes-Maxwell-Stefan systems is of great importance for an accurate
modeling and efficient numerical simulation of these applications. In this paper,
we make a step forward to this understanding by proving
the global-in-time existence of weak solutions and their long-time behavior
for Navier-Stokes-Maxwell-Stefan systems for incompressible fluids
under natural assumptions. This is the first analytical result for the full coupled
incompressible model.

More precisely, we consider a multicomponent fluid consisting of $N+1$ components
with the mass densities $\rho_i$, molar masses $M_i$, and velocities $u_i$.
As in \cite{BFS12},
we prescribe a system of partial mass balances together with a common mixture
momentum balance, where the diffusive fluxes are given by the
Maxwell-Stefan relations.
The partial mass balances for the molar concentrations $c_i=\rho_i/M_i$ read as
$$
  \pa_t c_i + \diver(j_i+c_i u) = 0, \quad i=1,\ldots,N+1,
$$
where the barycentric velocity $u$ and the total mass density $\rho^*$ of the fluid
are defined by
$\rho^* u=\sum_{i=1}^{N+1}\rho_i u_i$ and $\rho^*=\sum_{i=1}^{N+1}\rho_i$, and the
molar mass fluxes $j_i$ are given by $j_i=c_i(u_i-u)$. By definition of $j_i$,
it holds that $\sum_{i=1}^{N+1}M_i j_i=0$, and therefore,
one of the partial mass balances
can be replaced by the continuity equation $\pa_t\rho^* + \diver(\rho^* u)=0$.
The mixture momentum balance equations are
$$
  \pa_t(\rho^* u) + \diver(\rho^* u\otimes u-S) + \na p = \rho^* f,
$$
where $p$ is the pressure, the force density equals
$\rho^* f=\sum_{i=1}^{N+1}\rho_i f_i$, and the viscous stress tensor is
$S=\nu^*(\na u+\na u^\top)$, where $\nu^*$ is the viscosity constant.
In this paper, we suppose that $f_i=f$ and we impose
the incompressibility conditions
$$
  \rho^* = \mbox{const.}, \quad \diver u = 0.
$$
For simplicity, we set $\rho^*=1$ and $\nu^*=1$.

The above equations are closed by relating the molar mass fluxes $j_i$
to the molar concentrations $c_i$ by the Maxwell-Stefan equations
$$
  -\sum_{k=1}^{N+1}\frac{x_k j_i-x_i j_k}{D_{ik}} = c_i\na\mu_i - y_i\na p
	- \rho_i(f_i-f), \quad i=1,\ldots,N+1,
$$
where $x_i=c_i/c$ with $c=\sum_{k=1}^{N+1}c_i$ are the molar fractions,
$y_i=\rho_i/\rho^*=\rho_i$ are the mass fractions,
$\mu_i$ are the molar-based chemical
potentials, and $D_{ik}=D_{ki}>0$ for $i\neq k$ are the diffusion coefficients.
Our second assumption is that the mixture of gases
is ideal such that the chemical potentials
can be written as $\mu_i=\ln x_i+\mu_{0i}(p)$ with
$d\mu_{0i}/dp=\phi_i/c_i$, where
$\phi_i$ is the volume fraction (see \cite[Section 1.1]{BFS12}). Since $f_i=f$,
this implies that
$$
  -\sum_{k=1}^{N+1}\frac{x_k j_i-x_i j_k}{D_{ik}} = \frac{c_i}{x_i}\na x_i
	+ (\phi_i-y_i)\na p = c\na x_i + (\phi_i-y_i)\na p.
$$
We assume further that the volume and mass fractions are comparable such that
the contribution $(\phi_i-y_i)\na p$ can be neglected. This gives the desired
closure relations
\begin{equation*}
  -\sum_{k=1}^{N+1}\frac{x_k j_i-x_i j_k}{D_{ik}} = c\na x_i, \quad i=1,\ldots,N+1.
\end{equation*}
These relations, together with the mass balance equations, can also be
derived from a system of kinetic equations with BGK-type collision operator
in the Chapman-Enskog expansion \cite{AAP02}.

Setting $J_i=M_ij_i$, the incompressible Navier-Stokes-Maxwell-Stefan system
analyzed in this paper reads as
\begin{align}
  & \pa_t\rho_i + \diver(J_i + \rho_i u) = 0, \quad \mbox{in }\Omega,\ t>0,
	\label{1.rho} \\
	& \pa_t u + (u\cdot\na)u - \Delta u + \na p = f, \quad \diver u=0, \label{1.u} \\
  & \na x_i = -\sum_{k=1}^{N+1}\frac{\rho_k J_i-\rho_i J_k}{c^2 M_iM_k D_{ik}},
	\quad i=1,\ldots,N+1, \label{1.nax}
\end{align}
where $x_i$ and $\rho_i$ are related by $x_i=\rho_i/(cM_i)$ with
$c=\sum_{i=1}^{N+1}\rho_i/M_i$ and $\Omega\subset\R^d$ ($d\le 3$) is a
bounded domain. Note that $\rho^*=1$ implies that
$\sum_{i=1}^{N+1}\rho_i=\sum_{i=1}^{N+1}x_i=1$.
The initial and boundary conditions are
\begin{equation}\label{1.bic}
  \rho_i(\cdot,0)=\rho_i^0, \ u(\cdot,0)=u^0\quad\mbox{in }\Omega, \quad
  \na\rho_i\cdot\nu=0, \ u=0\quad\mbox{on }\pa\Omega,
\end{equation}
where $i=1,\ldots,N+1$ and $\nu$ is the normal exterior unit vector on $\pa\Omega$.

There are several difficulties to overcome in the analysis of the above system.

First, the molar mass fluxes are not explicitly given as a linear combination of
the mass density gradients, which makes necessary to invert the flux-gradient
relations \eqref{1.nax}. However, as the Maxwell-Stefan equations are
linearly dependent, we need to invert on a subspace. In the engineering
literature, this inversion is usually done in an approximate way \cite{APP03}.
Giovangigli \cite{Gio91} suggested an iterative procedure using the Perron-Frobenius
theory. A general inversion result was proved by Bothe \cite{Bot11},
again based on the Perron-Frobenius theory.

Second, equations \eqref{1.rho}-\eqref{1.nax} are strongly coupled through
the Maxwell-Stefan relations \eqref{1.nax} and the velocity $u$,
computed from the Navier-Stokes equations. Because of the cross-diffusion coupling
in \eqref{1.rho} and \eqref{1.nax}, standard tools like maximum principles and
regularity theory are not available. In particular, it is not clear how to prove
positive lower and upper bounds for the mass densities $\rho_i$ and even the
local existence of solutions is not trivial.

Third, we need to find suitable a priori estimates for the coupled system.
The energy method provides gradient estimates for the velocity,
but it is less clear how to derive estimates for the mass densities.
Moreover, the velocity does not need to be bounded such that the
term $\diver(\rho_i u)$ in \eqref{1.rho} needs to be treated carefully.

In view of these difficulties, it is not surprising that there exist
only partial results on such systems in the literature.
First results were concerned
with the Maxwell-Stefan equations \eqref{1.rho} and \eqref{1.nax}
with vanishing velocity $u=0$ and equal molar masses $M=M_i$.
Griepentrog \cite{Gri04} and later Bothe \cite{Bot11} derived a local existence
theory; Giovangigli \cite[Theorem 9.4.1]{Gio99} proved the global existence
of solutions with initial data sufficiently close to the equilibrium state;
Boudin, Grec, and Salvarani \cite{BGS12} investigated a particular two-component
model; and J\"ungel and Stelzer \cite{JuSt13} presented general global existence
results. The Maxwell-Stefan system with given bounded velovity $u\neq 0$ was
analyzed by Mucha, Pokorn\'y, and Zatorska \cite{MPZ13}. They imposed a special
diffusion matrix which avoids the inversion problem.

Other papers were concerned with the full coupled system but
in particular situations.
For instance, Zatorska \cite{Zat11} proved the existence
of weak solutions to the stationary compressible model with three fluid components
and special isobaric pressures. She also proved the sequential stability
of weak solutions to the two-component system on the three-dimensional torus
\cite{Zat12}. Mucha, Pokorn\'y, and Zatorska \cite{MPZ12}
showed a global existence result for a regularized compressible system for two
components. The Navier-Stokes equations contain artificial higher-order
differential operators which regularize the problem.
In \cite{MMT95}, the global existence for the incompressible
Navier-Stokes-Maxwell-Stefan system was announced but not proved.
For numerical approximations using a finite-volume method, we refer to
\cite{AMS98}.

In this paper, we prove a general global existence result for the
full coupled system
\eqref{1.rho}-\eqref{1.bic}, allowing for different molar masses $M_i$.
We overcome the above difficulties by combining
analytical tools for the incompressible Navier-Stokes equations due to
Temam \cite{Tem84}; the Perron-Frobenius theory for the matrix inversion problem
exploited by Bothe \cite{Bot11}; and the entropy-dissipation method developed
for cross-diffusion systems in \cite{ChJu04,JuSt13}. We detail our key ideas below.

In order to state our first main result, we introduce the following spaces
(see \cite[Chapter I]{Tem84}).
Let $\Omega\subset\R^d$ be a bounded domain with $\pa\Omega\in C^{1,1}$ and let
\begin{equation}\label{1.space}
\begin{aligned}
  \mathcal{H} &= \{u\in L^2(\Omega;\R^d):\diver u = 0,\
	u\cdot\nu|_{\pa\Omega}=0\}, \\
	\mathcal{V} &= \{u\in H_0^1(\Omega;\R^d):\diver u=0\},
	\quad \mathcal{V}_2 = \mathcal{V}\cap H^2(\Omega;\R^d), \\
	\widetilde H^2(\Omega;\R^N) &= \{q\in H^2(\Omega;\R^N):
	\na q\cdot\nu|_{\pa\Omega}=0\}.
\end{aligned}
\end{equation}
We define similarly the space $\widetilde H^2(\Omega)$.
We recall that functions $u\in L^2(\Omega;\R^d)$ with $\diver u\in L^2(\Omega)$
satisfy $u\cdot\nu|_{\pa\Omega}\in H^{-1/2}(\pa\Omega)$ such that the
space $\mathcal{H}$ is well defined \cite[Theorem I.1.2]{Tem84}.

\begin{theorem}[Global existence]\label{thm.ex}
Let $d=1,2,3$, $T>0$, and $D_{ij}=D_{ji}>0$ for $i,j=1,\ldots,N+1$, $i\neq j$.
Suppose that $f\in L^2(0,T;\mathcal{V}')$, $u^0\in \mathcal{H}$,
and let $\rho_1^0,\ldots,\rho_{N+1}^0\in L^1(\Omega)$ be nonnegative functions
which satisfy $\sum_{i=1}^{N+1}\rho_i^0=1$ and $h(\rho^0)<+\infty$, where
$\rho^0=(\rho_1^0,\ldots,\rho_N^0)$ and $h$ is defined in \eqref{1.h} below.
Then there exists a global weak
solution $(u,\rho_1,\ldots,\rho_{N+1})$ to \eqref{1.rho}-\eqref{1.bic}
(in the sense of \eqref{4.weak1}-\eqref{4.weak2} below)
such that $\rho_i\ge 0$, $\sum_{j=1}^{N+1}\rho_i=1$
in $\Omega\times(0,T)$, and
\begin{align*}
  & u\in L^\infty(0,T;\mathcal{H})\cap L^2(0,T;\mathcal{V}), \quad
	\pa_t u\in L^2(0,T;\mathcal{V}_2'), \\
	& \rho_i\in L^2(0,T;H^1(\Omega)), \quad
	\pa_t\rho_i\in L^2(0,T;\widetilde H^2(\Omega)'),
	\quad i=1,\ldots,N+1.
\end{align*}
\end{theorem}

We stress the fact that although the diffusion coefficients $D_{ij}$ are
constant, the diffusion matrix of the inverted Maxwell-Stefan system
(see \eqref{1.aux} below)
depends on the mass densities in a nonlinear way. Note that the same existence
result holds when we allow for reaction terms in \eqref{1.rho} which are
locally Lipschitz continuous and quasi-positive; see \cite{Bot11,JuSt13}.

The key ideas of the proof are as follows. First, we write \eqref{1.nax}
more compactly as $\na x=A(\rho)J$, where $x=(x_1,\ldots,x_{N+1})$,
$\rho=(\rho_1,\ldots,\rho_{N+1})$, $J=(J_1,\ldots,J_{N+1})$,
and $A(\rho)$ is a matrix.
Using the Perron-Frobenius theory, Bothe \cite{Bot11}
proved that $A(\rho)$ can be inverted on its image.
As in \cite{JuSt13}, it turns out that it is more convenient to work
with the system in $N$ components by eliminating the last equation
in \eqref{1.rho}. We set $x'=(x_1,\ldots,x_N)$ and similarly
for the other vectors. Then, inverting $\na x'=-A_0(\rho)J'$
(Lemma \ref{lem.A0}), \eqref{1.rho} becomes
\begin{equation}\label{1.aux}
  \pa_t\rho' + (u\cdot\na)\rho' - \diver(A_0(\rho)^{-1}\na x') = 0.
\end{equation}
This equation can be analyzed by exploiting its entropy structure. Indeed,
we associate to this system the entropy density (or, more precisely,
Gibbs free energy)
\begin{equation}\label{1.h}
  h(\rho') = c\sum_{i=1}^{N+1}x_i(\ln x_i-1)+c,
\end{equation}
where $\rho_{N+1}=1-\sum_{i=1}^N\rho_i$ is interpreted as a function of the
other mass densities. We ``symmetrize'' \eqref{1.aux} by introducing
the entropy variables
\begin{equation}\label{1.w}
  w_i = \frac{\pa h}{\pa\rho_i} = \frac{\ln x_i}{M_i} - \frac{\ln x_{N+1}}{M_{N+1}},
	\quad i=1,\ldots,N,
\end{equation}
and set $w=(w_1,\ldots,w_N)$.
The second equality in \eqref{1.w} is shown in Lemma \ref{lem.w} below.
Denoting by $D^2h(\rho')$
the Hessian of $h$ with respect to $\rho'$, \eqref{1.aux} is equivalent to
\begin{equation}\label{1.aux2}
   \pa_t\rho' + (u\cdot\na)\rho' - \diver(B(w)\na w) = 0,
\end{equation}
where $B(w)=A_0^{-1}(\rho')(D^2h)^{-1}(\rho')$ is symmetric and positive
definite (Lemma \ref{lem.hessian}). This formulation reveals the parabolic
structure of the equations. The mass density vector $\rho'$ is interpreted
as a function of $w$. If all molar masses are equal, $M_i=M$, this function
can be written as $\rho_i(w) =
\exp(Mw_i)(1+\sum_{j=1}^{N}\exp(Mw_j))^{-1}$ \cite{JuSt13}, showing that
\begin{equation}\label{1.bounds}
  0 < \rho' < 1 \quad\mbox{and}\quad \sum_{i=1}^{N}\rho_i<1.
\end{equation}
This formulation is no longer possible if the molar masses are different.
In this situation, $\rho'$ is implicitly given as a function of $w$; there is
no explicit formula anymore. However, we are able to show that the mapping
$\rho'\mapsto w$, defined by \eqref{1.w} and $x_i=\rho_i/(cM_i)$,
can be inverted and that \eqref{1.bounds} still holds (Corollary \ref{coro.w2rho}).

The entropy $H(\rho')=\int_\Omega h(\rho')dz$ provides suitable a priori estimates.
Indeed, using $w$ as a test function in \eqref{1.aux2}, a computation
(see Lemma \ref{lem.sqrtx} and the proof of Theorem \ref{thm.ex})
shows the entropy-dissipation inequality
\begin{equation}\label{1.entdiss}
  \frac{dH}{dt} = -\int_\Omega \na w:B(w)\na w dz
	\le -C_B\sum_{i=1}^{N+1}\int_\Omega|\na\sqrt{x_i}|^2 dz,
\end{equation}
where the constant $C_B>0$ only depends on the diffusion coefficients $D_{ij}$
and the molar masses $M_i$ and the double point ``:'' signifies summation
over both matrix indices. The key point is that the integral
$\int_\Omega ((u\cdot\na)\rho')\cdot w dz$ in \eqref{1.aux2}
vanishes (Lemma \ref{lem.ent}).
This yields $H^1$ estimates for $\sqrt{x_i}$ from which we conclude $H^1$ bounds
for $\rho_i$ (Lemma \ref{lem.est34}).
We note that a diffusion inequality which directly implies the above
entropy-dissipation inequality was first established in \cite[Section 4]{Gio99a}.

The proof of Theorem \ref{thm.ex} is based on a semi-discretization in time
of both the Navier-Stokes equations \eqref{1.u} and Maxwell-Stefan equations
\eqref{1.aux2} with time step $\tau>0$, together with a regularization
using the operator $\eps(\Delta^2 w+w)$ in \eqref{1.aux2},
which guarantees the
coercivity in $w$. The existence of a solution to the approximate problem
is shown by means of the Leray-Schauder fixed-point theorem. The discrete
analogon of the entropy-dissipation inequality \eqref{1.entdiss} provides bounds
uniform in the approximation parameters $\tau$ and $\eps$. By weak compactness
and the Aubin lemma, this allows us to perform the limit $(\tau,\eps)\to 0$.

System \eqref{1.rho}-\eqref{1.nax} admits the homogeneous steady state
$\bar\rho_i^0=\mbox{meas}(\Omega)^{-1}\|\rho_i^0\|_{L^1(\Omega)}$ or
$\bar x_i^0=\bar\rho_i^0/(\bar c^0 M_i)$, where
$\bar c^0=\sum_{i=1}^{N+1}\bar\rho_i^0/M_i$. We prove that the solution to
\eqref{1.rho}-\eqref{1.nax}, constructed in Theorem \ref{thm.ex},
converges exponentially fast to this stationary state.
For this, we introduce the relative entropy
\begin{equation}\label{1.relent}
  H^*(\rho) = \sum_{i=1}^{N+1}\int_\Omega cx_i\ln\frac{x_i}{\bar x_i^0}dz.
\end{equation}

\begin{theorem}[Exponential decay]\label{thm.long}
Let the assumptions of Theorem \ref{thm.ex} hold and let $f=0$.
We assume that there exists $0<\eta<1$ such that $\rho_i^0\ge\eta$
for $i=1,\ldots,N+1$. Let $(u,\rho)$ be
the weak solution constructed in Theorem \ref{thm.ex}. Then there
exist constants $C>0$, only depending on $\rho_i^0$ and $M_i$, and
$\lambda>0$, only depending on $\Omega$ and $M_i$, such that for all $t>0$
and $i=1,\ldots,N+1$,
$$
  \|x_i(\cdot,t)-\bar x_i^0\|_{L^1(\Omega)} \le Ce^{-\lambda t}\sqrt{H^*(\rho^0)},
$$
where $x_i=\rho_i/(cM_i)$ with $c=\sum_{i=1}^{N+1}\rho_i/M_i$ and
$\bar x_i^0=\rho_i^0/(\bar c^0M_i)$ with $\bar c^0=\sum_{i=1}^{N+1}\bar\rho_i^0/M_i$.
\end{theorem}

The proof is based on the entropy-dissipation inequality \eqref{1.entdiss}
by relating the
entropy dissipation with the entropy via the logarithmic Sobolev inequality
\cite[Remark 3.7]{Jue12}. Similarly as in \cite{JuSt13}, the difficulty of
the proof is that the approximate solution does not conserve the $L^1$ norm
because of the presence of the regularizing $\eps$-terms. The estimations
of these terms make the proof rather technical.

Compared to our previous work \cite{JuSt13}, the main novelties in this
paper are the coupling to the Navier-Stokes equations and the molar masses $M_i$
which are not equal. Because of the different molar
masses, we need to distinguish between the mass densities $\rho_i$ and the
molar fractions $x_i$, which makes necessary to derive some additional estimates.
In particular, the proof of the positive definiteness of the Hessian of $h$,
which implies the positive definiteness of $B(w)$, is rather involved
(see Lemma \ref{lem.hessian}).

The paper is organized as follows. In Section \ref{sec.prep}, we prove
some auxiliary results needed for the main proofs. In particular, we show
properties of the relations between $w$, $\rho$, and $x$ and of the
matrices $D^2 h$ and $B(w)$. The proofs of Theorems \ref{thm.ex}
and \ref{thm.long} are presented in Sections \ref{sec.ex} and \ref{sec.long},
respectively.


\section{Preparations}\label{sec.prep}

In this section, we show some auxiliary results which are used in the proofs
of the main theorems.

\subsection{Equivalent formulation of \eqref{1.rho} and \eqref{1.nax}}

We recall the notation $\rho=(\rho',\rho_{N+1})$, $\rho'=(\rho_1,\ldots,\rho_N)$
and similarly for $x$ and $J$, defined by $x_i=\rho_i/(cM_i)$
and $J_i=M_ij_i$ ($i=1,\ldots,N+1$). The matrix $\na\rho$ consists of the elements
$\pa\rho_i/\pa z_j$ ($1\le i\le N+1$, $1\le j\le d$), and we define similarly
$\na x$ and $\na w$. Then we can formulate \eqref{1.rho} and \eqref{1.nax}
more compactly as
\begin{equation}\label{3.aux1}
  \pa_t\rho + (u\cdot\na)\rho + \diver J = 0, \quad \na x = AJ,
\end{equation}
where the $(N+1)\times(N+1)$ matrix $A=A(\rho)=(A_{ij})$ is defined by
\begin{equation}\label{3.A}
\begin{aligned}
  A_{ij} &= d_{ij}\rho_i \quad\mbox{if }i\neq j,\ i,j=1,\ldots,N+1, \\
	A_{ij} &= -\sum_{k=1,\,k\neq i}^{N+1} d_{ik}\rho_k \quad\mbox{if }
	i=j=1,\ldots,N+1,
\end{aligned}
\end{equation}
and $d_{ij}=1/(c^2 M_iM_j D_{ij})$.
It is shown in \cite[Section 2]{JuSt13} that the system of
$N+1$ equations $\na x=AJ$
can be reduced to the first $N$ components, leading to
\begin{equation}\label{3.aux2}
   \pa_t\rho' + (u\cdot\na)\rho' + \diver J' = 0, \quad \na x' = A_0 J',
\end{equation}
where the $N\times N$ matrix $A_0=A_0(\rho')=(A_{ij}^0)$ is given by
\begin{equation}\label{3.A0}
\begin{aligned}
  A_{ij}^0 &= -(d_{ij}-d_{i,N+1})\rho_i \quad\mbox{if }i\neq j,\ i,j=1,\ldots,N, \\
	A_{ij}^0 &= \sum_{k=1,\,k\neq i}^N(d_{ik}-d_{i,N+1})\rho_k + d_{i,N+1}
	\quad\mbox{if }i=j=1,\ldots,N.
\end{aligned}
\end{equation}

\begin{lemma}\label{lem.A0}
The matrix $A_0$ is invertible and the elements of its inverse $A_0^{-1}$
are uniformly bounded in $\rho_1,\ldots,\rho_N\in[0,1]$.
\end{lemma}

\begin{proof}
The definition $c=\sum_{i=1}^{N+1}\rho_i/M_i$ and the property $0\le\rho_i\le 1$
imply that
\begin{equation}\label{3.cmax}
  \left(\max_{1\le i\le N+1}M_i\right)^{-1} \le c \le
	\left(\min_{1\le i\le N+1}M_i\right)^{-1}.
\end{equation}
Hence, the coefficients $d_{ij}=1/(c^2 M_i M_j D_{ij})$ are bounded uniformly
in $\rho_k\in[0,1]$. Therefore, the proof of Lemma 2.3 in \cite{JuSt13}
applies, proving the result.
\end{proof}


\subsection{Entropy variables}\label{sec.entvar}

We recall the
relations $x_i=\rho_i/(cM_i)$, $c=\sum_{i=1}^{N+1}\rho_i/M_i$, and
$\sum_{i=1}^{N+1}\rho_i=1$. Since
$$
  x_{N+1} = \frac{\rho_{N+1}}{cM_{N+1}}
	= \frac{1}{cM_{N+1}}\left(1-\sum_{i=1}^N\rho_i\right), \quad
	c = \sum_{i=1}^N\frac{\rho_i}{M_i} + \frac{1}{M_{N+1}}
	\left(1-\sum_{i=1}^N\rho_i\right),
$$
we may interpret the entropy density \eqref{1.h} as a function of
$\rho'=(\rho_1,\ldots,\rho_N)$, which gives
\begin{align}
  h(\rho') &= c\sum_{i=1}^{N}x_i(\ln x_i-1)
	+ cx_{N+1}(\ln x_{N+1}-1) + c \nonumber \\
  &= \sum_{i=1}^{N}\frac{\rho_i}{M_i}\left(\ln\frac{\rho_i}{M_i}-1\right)
	+ \frac{\rho_{N+1}}{M_{N+1}}\left(\ln\frac{\rho_{N+1}}{M_{N+1}}-1\right)
	- c(\ln c-1), \label{3.hrho}
\end{align}
First, we prove that the entropy variables can be written as in \eqref{1.w}.

\begin{lemma}\label{lem.w}
The entropy variables $w_i=\pa h(\rho')/\pa\rho_i$ are given by
\begin{equation}\label{3.w}
  w_i = \frac{\ln x_i}{M_i} - \frac{\ln x_{N+1}}{M_{N+1}}, \quad i=1,\ldots,N.
\end{equation}
\end{lemma}

\begin{proof}
The proof is just a computation. Indeed, we infer from
$$
  \frac{\pa c}{\pa\rho_i} = \frac{\pa}{\pa\rho_i}\left(\sum_{k=1}^N\frac{\rho_k}{M_k}
	+ \frac{1}{M_{N+1}}\left(1-\sum_{k=1}^N\rho_k\right)\right)
	= \frac{1}{M_i}-\frac{1}{M_{N+1}}
$$
for $i=1,\ldots,N$ that
\begin{align*}
  \frac{\pa h}{\pa \rho_i}(\rho') &= \frac{1}{M_i}\ln\frac{\rho_i}{M_i}
	- \frac{1}{M_{N+1}}\ln\frac{\rho_{N+1}}{M_{N+1}}
	- \left(\frac{1}{M_i}-\frac{1}{M_{N+1}}\right)\ln c \\
	&= \frac{1}{M_i}\ln\frac{\rho_i}{cM_i}
	- \frac{1}{M_{N+1}}\ln\frac{\rho_{N+1}}{cM_{N+1}},
\end{align*}
and since $\rho_i/(cM_i)=x_i$, the conclusion follows.
\end{proof}

We claim that we can invert the mapping $x'\mapsto w$, defined by \eqref{3.w}.

\begin{lemma}\label{lem.w2x}
Let $w=(w_1,\ldots,w_N)\in\R^N$ be given. Then there exists a unique
$(x_1,\ldots,x_N)\in(0,1)^N$ satisfying $\sum_{i=1}^N x_i<1$
such that \eqref{3.w} holds with $x_{N+1}=1-\sum^N_{i=1}x_i$.
In particular, the mapping
$\R^N\to(0,1)^N$, $x'(w)=(x_1,\ldots,x_N)$, is bounded.
\end{lemma}

\begin{proof}
Introduce the function $f(s)=\sum_{i=1}^N (1-s)^{M_i/M_{N+1}}\exp(M_i w_i)$ for
$s\in [0,1]$. Then $f$ is strictly decreasing in $[0,1]$ and
$0=f(1)<f(s)<f(0)=\sum_{i=1}^N\exp(M_i w_i)$ for $s\in (0,1)$. By continuity,
there exists a unique fixed point $s_0\in(0,1)$, $f(s_0)=s_0$. Defining
$x_i=(1-s_0)^{M_i/M_{N+1}}\exp(M_i w_i)$ for $i=1,\ldots,N$, we infer that
$x_i>0$ and $\sum_{i=1}^N x_i=f(s_0)=s_0<1$.
Hence, in view of $x_{N+1}=1-s_0$, \eqref{3.w} holds.
\end{proof}

Given $\rho$, we can define $x_i=\rho_i/(cM_i)$, where
$c=\sum_{i=1}^{N+1}\rho_i/M_i$. The following lemma ensures that this mapping
is invertible.

\begin{lemma}\label{lem.x2rho}
Let $x'=(x_1,\ldots,x_{N})\in(0,1)^{N}$ and $x_{N+1}=1-\sum_{i=1}^Nx_i>0$ be given
and define for $i=1,\ldots,N+1$,
$$
  \rho_i(x') = \rho_i = cM_ix_i, \quad\mbox{where }
	c=\frac{1}{\sum_{k=1}^{N+1}M_kx_k}.
$$
Then $(\rho_1,\ldots,\rho_{N})\in(0,1)^{N}$ is the unique vector satisfying
$\rho_{N+1}=1-\sum_{i=1}^N\rho_i>0$,
$x_i=\rho_i/(cM_i)$ for $i=1,\ldots,N+1$, and $c=\sum_{k=1}^{N+1}\rho_k/M_k$.
\end{lemma}

The proof follows immediately from
$\sum_{k=1}^{N+1}\rho_k/M_k = \sum_{k=1}^{N+1}cx_k = c$, and the fact that
${\rho_i}/{M_ix_i}=\sum_{k=1}^{N+1}{\rho_k}/{M_k}$ for $i=1,2,...,N$
have unique solutions $\rho_i={M_ix_i}/{\sum_{k=1}^{N+1}M_kx_k}$ for
$i=1,2,...,N$ by applying Cramer's rule.

Combining Lemmas \ref{lem.w2x} and \ref{lem.x2rho}, we infer the following result.

\begin{corollary}\label{coro.w2rho}
Let $w=(w_1,\ldots,w_N)\in\R^N$ be given. Then there exists a unique vector
$(\rho_1,\ldots,\rho_N)\in(0,1)^N$ satisfying $\sum_{i=1}^N\rho_i<1$ such that
\eqref{3.w} holds for $\rho_{N+1}=1-\sum_{i=1}^N\rho_i$ and
$x_i=\rho_i/(cM_i)$ with $c=\sum_{i=1}^{N+1}\rho_i/M_i$.
Moreover, the mapping $\R^N\to(0,1)^N$, $\rho'(w)=(\rho_1,\ldots,\rho_N)$,
is bounded.
\end{corollary}


\subsection{Hessian of the entropy density}\label{sec.hess}

We prove some properties of the Hessian
$(H_{ij})=(\pa^2 h(\rho')/\pa\rho_i\pa\rho_j)_{1\leq i,j\leq N}
=(\pa w_i/\pa\rho_j)_{1\leq i,j\leq N}$ and the matrix
$(G_{ij})=(\pa w_i/\pa x_j)_{1\leq i,j\leq N}$. Differentiating
\eqref{3.hrho} gives
$$
  H_{ij}
	= \frac{\delta_{ij}}{M_i\rho_i} + \frac{1}{M_{N+1}\rho_{N+1}}
	- \frac{1}{c}\left(\frac{1}{M_i}-\frac{1}{M_{N+1}}\right)
	\left(\frac{1}{M_j}-\frac{1}{M_{N+1}}\right), \quad i,j=1,\ldots,N,
$$
where $\delta_{ij}$ denotes the Kronecker delta.

\begin{lemma}\label{lem.hessian}
The matrix $(H_{ij})$ is symmetric and positive definite for all
$\rho_1,\ldots,\rho_N>0$ satisfying $\sum_{i=1}^N\rho_i<1$.
\end{lemma}

\begin{proof}
We claim that the principal minors $\det H_k$ of $(H_{ij})$ satisfy
\begin{equation}\label{3.minor}
  \det H_k > \frac{2}{cM_{N+1}\prod_{\ell=1}^k M_\ell}\left(
	\sum_{i,j=1,\,i<j}^k\frac{1}{\rho_{N+1}\prod_{\ell=1,\,\ell\neq i,j}^k\rho_\ell}
	+ \sum_{j=1}^k \frac{1}{\prod_{\ell=1,\,\ell\neq j}^k\rho_\ell}\right) > 0
\end{equation}
for $k=1,\ldots,N$. Then the positive definiteness of $(H_{ij})$ follows
from Sylvester's criterion.
It remains to prove \eqref{3.minor}. Since each column of $H_k$ can be written
for $j=1,2,...,k,$ as the difference
$$
  \begin{pmatrix}
	\delta_{1j}(M_1\rho_1)^{-1}+(M_{N+1}\rho_{N+1})^{-1} \\
	\vdots \\
	\delta_{kj}(M_k\rho_k)^{-1}+(M_{N+1}\rho_{N+1})^{-1}
	\end{pmatrix}
	- \frac{1}{c}\left(\frac{1}{M_j}-\frac{1}{M_{N+1}}\right)
	\begin{pmatrix}
	M_1^{-1}-M_{N+1}^{-1} \\
	\vdots \\
	M_k^{-1}-M_{N+1}^{-1}
	\end{pmatrix},
$$
a calculation shows that
\begin{align*}
  \det H_k &= \frac{1}{\prod_{\ell=1}^k M_\ell\rho_\ell}
	\left(\sum_{j=1}^k\frac{M_j \rho_j}{M_{N+1}\rho_{N+1}}+1\right) \\
	&\phantom{xx}{}- \frac{1}{c}\sum_{j=1}^k\left(\frac{1}{M_j}-\frac{1}{M_{N+1}}\right)
	\frac{1}{\prod_{\ell=1,\,\ell\neq j}^k M_\ell\rho_\ell} \\
	&\phantom{xx}{}\times\left(\sum_{i=1,\,i\neq j}^k\frac{M_i\rho_i}{M_{N+1}\rho_{N+1}}
	\left(\frac{1}{M_j}-\frac{1}{M_i}\right)
	+ \left(\frac{1}{M_j}-\frac{1}{M_{N+1}}\right)\right).
\end{align*}
Multiplying this expression by $c$ and rearranging the terms, we find that
\begin{align*}
  c\det H_k &= \left(\sum_{j=1}^k\frac{c}{M_{N+1}\rho_{N+1}
	\prod^{k}_{\ell=1,\,\ell\neq j}M_\ell\rho_\ell}
	+ \frac{c}{\prod_{\ell=1}^k M_\ell\rho_\ell}\right) \\
	&\phantom{xx}{}- \sum_{j=1}^k\left(\frac{1}{M_j}-\frac{1}{M_{N+1}}\right)^2
	\frac{1}{\prod_{\ell=1,\,\ell\neq j}^k M_\ell\rho_\ell} \\
	&\phantom{xx}{}-\sum_{j=1}^k\left(\frac{1}{M_j}-\frac{1}{M_{N+1}}\right)
	\sum_{i=1,\,i\neq j}^k\left(\frac{1}{M_j}-\frac{1}{M_i}\right)
	\frac{M_i\rho_i}{M_{N+1}\rho_{N+1}\prod_{\ell=1,\,\ell\neq j}^k M_\ell\rho_\ell} \\
	&= I_1+I_2+I_3.
\end{align*}
Recalling that $c=\sum_{\ell=1}^{N+1}\rho_\ell/M_\ell$, we can estimate as follows:
\begin{align*}
  I_1 &> \sum_{j=1}^k\frac{\sum_{i=1,\,i\neq j}^k\rho_i/M_i + \rho_{N+1}/M_{N+1}}{
	M_{N+1}\rho_{N+1}\prod_{\ell=1,\,\ell\neq j}^kM_\ell\rho_\ell}
	+ \frac{\sum_{j=1}^k\rho_j/M_j}{\prod_{\ell=1}^k M_\ell\rho_\ell} \\
	&= \sum_{j=1}^k \left(\sum_{i=1,\,i\neq j}^k\frac{1}{M_i^2 M_{N+1}\rho_{N+1}
	\prod_{\ell=1,\,\ell\neq i,j}^k M_\ell\rho_\ell}
	+ \frac{1}{M_{N+1}^2\prod_{\ell=1,\,\ell\neq j}^k M_\ell\rho_\ell}\right) \\
	&\phantom{xx}{}
	+ \sum_{j=1}^k\frac{1}{M_j^2\prod_{\ell=1,\,\ell\neq j}^k M_\ell\rho_\ell} \\
  &= \sum_{i,j=1,\,i<j}^k\frac{M_i^{-2}+M_j^{-2}}{M_{N+1}\rho_{N+1}
	\prod^k_{\ell=1,\,\ell\neq i,j}M_\ell\rho_\ell}
	+ \sum_{j=1}^k \frac{M_j^{-2}+M_{N+1}^{-2}}{\prod_{\ell=1,\,\ell\neq j}^k
	M_\ell\rho_\ell}.
\end{align*}
Using $\sum_{i,j}b_{ij}(a_j-a)(a_j-a_i)=\sum_{i<j}b_{ij}(a_j-a_i)^2$
for numbers $a$, $a_i\in\R$ and $b_{ij}=b_{ji}\in\R$,
the last term $I_3$ can be formulated as
$$
  I_3 = -\sum_{j=1}^k\sum_{i=1,\,i\neq j}^k\frac{(M_j^{-1}-M_{N+1}^{-1})
	(M_j^{-1}-M_i^{-1})}{M_{N+1}\rho_{N+1}\prod^k_{\ell=1,\,\ell\neq i,j}
	M_\ell\rho_\ell}
	= -\sum_{i,j=1,\,i<j}^k\frac{(M_i^{-1}-M_j^{-1})^2}{M_{N+1}\rho_{N+1}
	\prod_{\ell=1,\,\ell\neq i,j}^k M_\ell\rho_\ell}.
$$
Therefore, we infer that
\begin{align*}
  c\det H_k &> \sum_{i,j=1,\,i<j}^k\frac{2M_i^{-1}M_j^{-1}}{M_{N+1}\rho_{N+1}
	\prod^k_{\ell=1,\,\ell\neq i,j}M_\ell\rho_\ell}
	+ \sum_{j=1}^k\frac{2M_j^{-1}M_{N+1}^{-1}}{\prod_{\ell=1,\,\ell\neq j}^k
	M_\ell\rho_\ell} \\
	&= \frac{2}{M_{N+1}\prod_{\ell=1}^k M_\ell}\left(\sum_{i,j=1,\,i<j}^k
	\frac{1}{\rho_{N+1}\prod_{\ell=1,\,\ell\neq i,j}^k \rho_\ell}
	+ \sum_{j=1}^k\frac{1}{\prod_{\ell=1,\,\ell\neq j}^k \rho_\ell}\right),
\end{align*}
and \eqref{3.minor} follows.
\end{proof}

The coefficients $G_{ij}=\pa w_i/\pa x_j$ are given by
\begin{equation}\label{3.G}
  G_{ij} = \frac{1}{M_{N+1}x_{N+1}} + \frac{\delta_{ij}}{M_i x_i}
	= c\left(\frac{1}{\rho_{N+1}} + \frac{\delta_{ij}}{\rho_i}\right),
	\quad i,j=1,\ldots,N,
\end{equation}
since $x_i=\rho_i/(cM_i)$. We recall that $w(\rho')$
is computed in \eqref{3.w}.

\begin{lemma}\label{lem.G}
It holds for all $\rho_1,\ldots,\rho_N>0$ satisfying
$\rho_{N+1}=1-\sum_{i=1}^N\rho_i>0$:
\begin{enumerate}[\rm (i)]
\item The matrix $G(\rho')=(G_{ij})$ and its inverse $G^{-1}(\rho')$
are positive definite.
\item $\na w(\rho')=G(\rho')\na x'(\rho')$.
\item The elements of the $N\times N$ matrix $d\rho'/dx'=(\pa\rho_i/\pa x_k)$
are bounded by a constant which depends only on the molar masses $M_i$.
\item The $N\times N$ matrix $B(\rho')=A_0^{-1}(\rho')G^{-1}(\rho')$ is
symmetric, positive definite, and its elements are uniformly bounded.
\end{enumerate}
\end{lemma}

\begin{proof}
(i) The explicit expression \eqref{3.G} shows that $G(\rho')$ is symmetric.
Since all principal minors $\det G_k$ of $G(\rho')$,
$$
  \det G_k = \frac{\sum_{i=1}^k M_i x_i + M_{N+1}x_{N+1}}{(\prod_{i=1}^k M_i x_i)
	M_{N+1}x_{N+1}}, \quad k=1,\ldots,N,
$$
are positive, Sylvester's criterion implies that $G(\rho')$ is positive definite.
Consequently, also $G^{-1}(\rho')$ is positive definite.

(ii)
We infer from \eqref{3.w} that
$$
  \nabla{w}_i=\frac{\nabla{x}_i}{M_ix_i}+\sum_{j=1}^N
  \frac{\nabla{x}_j}{M_{N+1}x_{N+1}}
  =\sum_{j=1}^N G_{ij}{\nabla{x}_j},\quad i=1,2,\ldots,N,
$$
and hence $\nabla{w}={G}(\rho')\nabla {x}'$.

(iii) By Lemma \ref{lem.x2rho}, it follows that
\begin{equation}\label{3.drhodx}
  \frac{\pa\rho_i}{\pa x_k}
	= cM_i\delta_{ik} - c^2M_ix_i(M_k-M_{N+1}), \quad i,k=1,\ldots,N,
\end{equation}
where $c=1/\sum_{j=1}^{N+1}M_jx_j$. The claim follows from the inequalities
$0<x_i<1$ and the bounds \eqref{3.cmax}.

(iv) We set $G(\rho')=c K(\rho')$, where the elements $K_{ij}$ of $K(\rho')$
are given by $K_{ij}=1/\rho_{N+1}+\delta_{ij}/\rho_i$ for $i,j=1,\ldots,N$.
In view of part (i) of the proof,
the matrix $K(\rho')$ is symmetric and positive definite, hence invertible.
Then, by Lemma 2.4 in \cite{JuSt13}, $A_0^{-1}(\rho')K^{-1}(\rho')$ is
symmetric and positive definite and its elements are uniformly bounded.
Consequently, the same holds for $B(\rho')=c^{-1}A_0^{-1}(\rho')
K^{-1}(\rho')$. This ends the proof.
\end{proof}

From Lemma \ref{lem.G} follows that
\begin{equation}\label{3.Bnaw}
  A_0^{-1}(\rho')\na x'(\rho') = \big(A_0^{-1}(\rho')G^{-1}(\rho')\big)
	(G(\rho')\na x'(\rho')) = B(\rho')\na w(\rho').
\end{equation}
We have shown at the end of Section \ref{sec.entvar} that $\rho'$ can be
interpreted as a function of $w$. Therefore, setting $B(w):=B(\rho'(w))$,
\eqref{1.aux} can be written as
\begin{equation}\label{3.eq.w}
  \pa_t\rho'(w) + (u\cdot\na)\rho'(w) - \diver(B(w)\na w) = 0.
\end{equation}
The boundary conditions are given by
\begin{equation}\label{3.bc}
  \na w_i\cdot\nu = 0\quad\mbox{on }\pa\Omega,\ t>0, \quad i=1,\ldots,N,
\end{equation}
since $\na\rho_j\cdot\nu=0$ on $\pa\Omega$ for all $j$ implies that
$$
  \na x_i\cdot\nu = \na\frac{\rho_i}{cM_i}\cdot\nu
	= \frac{\na\rho_i\cdot\nu}{cM_i}
	- \sum_{j=1}^{N+1}\frac{\rho_i\na\rho_j\cdot\nu}{c^2M_iM_j} = 0
$$
and thus $\na w_i\cdot\nu=(G(\rho')\na x)_i\cdot\nu = 0$ on $\pa\Omega$.


\subsection{Some estimates}

We show two results which are needed in the proof of the existence theorem.

\begin{lemma}\label{lem.ent}
Let $u\in\mathcal{V}$ and $w\in H^1(\Omega)$. Then
$$
  \int_\Omega((u\cdot\na)\rho'(w))\cdot wdz = 0.
$$
\end{lemma}

\begin{proof}
Using $\diver u=0$, the characterization \eqref{3.w} of $w_i$,
and $\rho_i/M_i=cx_i$, we obtain after an integration by parts,
\begin{align*}
  \int_\Omega((u\cdot\na)\rho'(w))\cdot w dz
	&= \sum_{i=1}^N\int_\Omega (u\cdot\na\rho_i(w))w_i dz
	= -\sum_{i=1}^N\int_\Omega (u\cdot\na w_i)\rho_i(w) dz \\
	&= -\sum_{i=1}^N\int_\Omega \rho_i(w)u\cdot\left(\frac{\na x_i}{M_ix_i}
	- \frac{\na x_{N+1}}{M_{N+1}x_{N+1}}\right)dz \\
	&= -\sum_{i=1}^N\int_\Omega cu\cdot\na x_i dz
	+ \int_\Omega\sum_{i=1}^N\rho_i \frac{u\cdot\na x_{N+1}}{M_{N+1}x_{N+1}}dz.
\end{align*}
Because of $\sum_{i=1}^N \rho_i=1-\rho_{N+1}$ and $\rho_{N+1}/(M_{N+1}x_{N+1})=c$,
the last integral equals
\begin{align*}
  \int_\Omega \frac{1-\rho_{N+1}}{M_{N+1}x_{N+1}}u\cdot\na x_{N+1}dz
	&= \frac{1}{M_{N+1}}\int_\Omega u\cdot\na(\ln x_{N+1})dz
	- \int_\Omega cu\cdot\na x_{N+1}dz \\
	&= \sum_{i=1}^N\int_\Omega cu\cdot\na x_i dz,
\end{align*}
where we integrated by parts and used $\diver u=0$ and $x_{N+1}=1-\sum_{i=1}^N x_i$.
This shows the lemma.
\end{proof}

In the following, we employ the notation $f(x)=(f(x_1),\ldots,f(x_{N+1}))$
for vectors $x=(x_1,\ldots,x_{N+1})$ and arbitrary functions $f$.

\begin{lemma}\label{lem.sqrtx}
Let $w\in H^1(\Omega)$. Then there exists a constant $C_B>0$, only
depending on the coefficients $D_{ij}$ and $M_i$ such that
$$
  \int_\Omega \na w:B(w)\na wdz \ge C_B\int_\Omega|\na\sqrt{x}|^2 dz.
$$
\end{lemma}

\begin{proof}
We follow the proof of Lemma 3.2 in \cite{JuSt13}. In contrast to that proof,
we have to take into account the different molar masses $M_i$ which complicates
the analysis. First, we claim that
$$
  \na w:B(w)\na w = \na s:(-\widetilde A)^{-1}\na x,
$$
where $s=(\ln x_1/M_1,\ldots,\ln x_{N+1}/M_{N+1})$ and
$\widetilde A=A|_{\text{im}(A)}$. To prove this claim, we set
$r'=(r_1,\ldots,r_N)^\top
= B(w)\na w\in\R^{N\times d}$ and $r_{N+1}=-\sum_{i=1}^N r_i\in\R^d$. Then,
by \eqref{3.w},
\begin{equation}\label{3.aux4}
  \na w:B(w)\na w = \sum_{i=1}^N\left(\frac{\na\ln x_i}{M_i}
	-\frac{\na\ln x_{N+1}}{M_{N+1}}
	\right)\cdot r_i = \sum_{i=1}^{N+1}\frac{\na \ln x_i}{M_i}\cdot r_i = \na s:r,
\end{equation}
where $r=(r',r_{N+1})^\top$. By \eqref{3.Bnaw}, $\na x'=A_0r'$, and the definitions
\eqref{3.A} and \eqref{3.A0} of $A$ and $A_0$, respectively,
we obtain for $i=1,\ldots,N$,
$$
  \na x_i = \sum_{j=1,\,j\neq i}^N(d_{ij}-d_{i,N+1})
	(\rho_j r_i^\top-\rho_i r_j^\top)
	+ d_{i,N+1}r_i^\top = (-Ar)_i = (-\widetilde Ar)_i,
$$
since $\mbox{im}(A)=(\mbox{span}(1,\ldots,1))^\perp$ and each column of $r$
is an element of $\mbox{im}(A)$.
Moreover, each column of $\widetilde Ar$ is also an element of $\mbox{im}(A)$,
so that
$$
  (-\widetilde Ar)_{N+1} = -\sum_{i=1}^N(-\widetilde Ar)_i
	= -\sum_{i=1}^N\na x_i = \na x_{N+1}.
$$
Therefore, $\na x = -\widetilde Ar$. It is shown in \cite[Lemma 2.2]{JuSt13}
that $\widetilde A$ is invertible. Thus, $r=(-\widetilde A)^{-1}\na x$, and
inserting this expression into \eqref{3.aux4} proves the claim.

Next, we introduce the symmetric matrix
$\widetilde A_S=P^{-1/2}\widetilde AP^{1/2}$, where
$$
  P^{1/2} = \mbox{diag}((M_1x_1)^{1/2},\ldots,(M_{N+1}x_{N+1})^{1/2}).
$$
Then $(-\widetilde A_S)^{-1}=P^{-1/2}(-\widetilde A)^{-1}P^{1/2}$.
Arguing similarly as in \cite[Lemma 2.2]{JuSt13}, we find that
$(-\widetilde A_S)^{-1}$
is a self-adjoint endomorphism whose smallest eigenvalue is bounded from below
by some positive constant, say $C_0>0$, which depends only on $(D_{ij})$.
This gives
\begin{align*}
  \na w & :B(w)\na w = \na s:(-\widetilde A)^{-1}\na x \\
	&= 4\na\sqrt{x}:\mbox{diag}\big(M_1^{-1}x_1^{-1/2},\ldots,
	M_{N+1}^{-1}x_{N+1}^{-1/2}\big)
	(-\widetilde A)^{-1}\mbox{diag}\big(x_1^{1/2},\ldots,x_{N+1}^{1/2}\big)
	\na\sqrt{x} \\
	&= 4\na\sqrt{x}:\left(\mbox{diag}\big(M_1^{-1}x_1^{-1/2},\ldots,
	M_{N+1}^{-1}x_{N+1}^{-1/2}\big)P^{1/2}\right)
	(P^{-1/2}(-\widetilde A)^{-1}P^{1/2}) \\
	&\phantom{xx}{}\times\left(P^{-1/2}\mbox{diag}
	\big(x_1^{1/2},\ldots,x_{N+1}^{1/2}\big)\right)\na\sqrt{x} \\
	&= 4\na\sqrt{x}:\mbox{diag}\big(M_1^{-1/2},\ldots,M_{N+1}^{-1/2}\big)
	(-\widetilde A_S)^{-1}\mbox{diag}\big(M_1^{-1/2},\ldots,M_{N+1}^{-1/2}\big)
	\na\sqrt{x} \\
	&\ge C_0\big|\mbox{diag}\big(M_1^{-1/2},\ldots,M_{N+1}^{-1/2}\big)
	\na\sqrt{x}\big|^2 \\
	&\ge C_B|\na\sqrt{x}|^2,
\end{align*}
where $C_B=C_0(\max_{1\le i\le N+1}M_i)^{-1/2}$.
\end{proof}


\section{Proof of Theorem \ref{thm.ex}}\label{sec.ex}

We say that $(u,\rho)$ is a weak solution
to \eqref{1.rho}-\eqref{1.bic} if for any
$v\in C_0^\infty(\Omega\times[0,T);\R^d)$ with $\diver v=0$,
\begin{align}
  -\int_0^T\int_\Omega & u\cdot\pa_t v dz\,dt
	+ \int_0^T\int_\Omega((u\cdot\na)u)\cdot v dz\,dt
	+ \int_0^T\int_\Omega \na u:\na v dz\,dt \nonumber \\
	&= \int_0^T\langle f,v\rangle dt + \int_\Omega u^0\cdot v(\cdot,0)dz,
	\label{4.weak1}
\end{align}
where $\langle\cdot,\cdot\rangle$ denotes the duality pairing between
$\mathcal{V}'$ and $\mathcal{V}$; and if for any
$q\in C^\infty_0(\overline{\Omega}
\times[0,T);\R^N)$ with $\na q\cdot\nu|_{\pa\Omega}=0$,
\begin{align}
  -\int_0^T\int_\Omega & \rho'\cdot\pa_t q dz\,dt
	+ \int_0^T\int_\Omega\na q:A_0^{-1}(\rho')\na x'(\rho')dz\,dt
	+ \int_0^T\int_\Omega((u\cdot\na)\rho')\cdot q dz\,dt \nonumber \\
	&= \int_\Omega (\rho^0)'\cdot q(\cdot,0)dz. \label{4.weak2}
\end{align}
The proof of Theorem \ref{thm.ex} is divided into several steps.

\subsection{Approximate problem}

Let $M\in\N$ and set $\tau=T/M$. Let $k\in\{1,\ldots,M\}$. Given
$(u^{k-1},w^{k-1})\in\mathcal{V}\times L^{\infty}(\Omega;\R^N)$,
we solve a regularized approximate problem for
\eqref{1.rho}-\eqref{1.nax}: For any $v\in\mathcal{V}$ and
$q\in \widetilde H^2(\Omega;\R^N)$:
\begin{align}
  & \int_\Omega\frac{u^k-u^{k-1}}{\tau}\cdot vdz
	+ \int_\Omega((u^{k-1}\cdot\na)u^k)\cdot v dz
	+ \int_\Omega\na u^k:\na vdz
	= \int_\Omega f^k\cdot vdz, \label{4.weak.u} \\
	& \int_\Omega\frac{\rho'(w^k)-\rho'(w^{k-1})}{\tau}\cdot q dz
	+ \int_\Omega\na q:A_0^{-1}(\rho'(w^k))\na x'(\rho'(w^k))dz \label{4.weak.rho} \\
	&\phantom{xx}{}+ \int_\Omega((u^k\cdot\na)\rho'(w^k))\cdot q dz
	+ \eps\int_\Omega(\Delta w^k\cdot\Delta q + w^k\cdot q)dz
	= 0, \nonumber
\end{align}
where $f^k=\tau^{-1}\int_{(k-1)\tau}^{k\tau}f(\cdot,t)dt\in\mathcal{V}'$
and $\rho'(w^k)$ is defined in Corollary \ref{coro.w2rho}.
Because of \eqref{3.Bnaw}, equation \eqref{4.weak.rho} is equivalent to
\begin{align}
	& \int_\Omega\frac{\rho'(w^k)-\rho'(w^{k-1})}{\tau}\cdot q dz
	+ \int_\Omega\na q:B(w^k)\na w^k dz \nonumber \\
	&\phantom{xx}{}+ \int_\Omega((u^k\cdot\na)\rho'(w^k))\cdot q dz
	+ \eps\int_\Omega(\Delta w^k\cdot\Delta q^k + w^k\cdot q)dz
	= 0, \label{4.weak.w}
\end{align}

Define for $0<\eta<1$ the space of bounded, strictly positive functions
$$
  Y_\eta = \Big\{q=(q_1,\ldots,q_N)\in L^\infty(\Omega;\R^N):
	q_i\ge\eta\mbox{ for }i=1,\ldots.N,\
	q_{N+1}=1-\sum_{i=1}^N q_i\ge\eta\Big\}.
$$

\begin{lemma}\label{lem.approx}
Let $\eta^{k-1}\in(0,1)$ and
$(u^{k-1},\rho^{k-1})\in\mathcal{V}\times Y_{\eta^{k-1}}$ with
$\rho^{k-1}=\rho'(w^{k-1})$.
Then there exist $\eta^k\in(0,1)$ and $(u^k,w^k)\in\mathcal{V}
\times\widetilde H^2(\Omega;\R^N)$
which solves \eqref{4.weak.u}-\eqref{4.weak.rho} satisfying $\rho'(w^k)\in
Y_{\eta^k}$.
\end{lemma}

\begin{proof}
{\em Step 1.} By standard theory of the incompressible Navier-Stokes equations
\cite{Tem84}, there exists a unique solution $u\in\mathcal{V}$ to
\begin{equation}\label{4.LM}
  a_1(u,v) = F_1(v)\quad\mbox{for }v\in\mathcal{V},
\end{equation}
where for $u$, $v\in\mathcal{V}$,
\begin{align*}
  a_1(u,v) &= \frac{1}{\tau}\int_\Omega u\cdot vdz
	+ \int_\Omega((u^{k-1}\cdot\na)u)\cdot v dz
	+ \int_\Omega\na u:\na vdz, \\
	F_1(v) &= \frac{1}{\tau}\int_\Omega u^{k-1}\cdot vdz + \int_\Omega f^k\cdot v dz.
\end{align*}
Indeed, since $H^1(\Omega)\hookrightarrow L^4(\Omega)$ for $d\le 4$, we have
for some (generic) constant $C>0$,
$$
  \left|\int_\Omega((u^{k-1}\cdot\na)u)\cdot v dz\right|
	\le \|u^{k-1}\|_{L^4(\Omega)}\|\na u\|_{L^2(\Omega)}\|v\|_{L^4(\Omega)}
	\le C\|u\|_{H^1(\Omega)}\|v\|_{H^1(\Omega)},
$$
and using $\diver u^{k-1}=0$, it follows that
$$
  a_1(u,u) = \frac{1}{\tau}\int_\Omega|u|^2 dz + \int_\Omega\|\na u\|^2 dz
	\ge C\|u\|_{H^1(\Omega)}^2.
$$
Thus, $a_1(\cdot,\cdot)$ is a bounded, coercive bilinear form on
$\mathcal{V}$ and $F_1\in\mathcal{V}'$. By Lax-Milgram's lemma,
there exists a unique solution $u\in\mathcal{V}$ to \eqref{4.LM}.

{\em Step 2.} Let $u\in\mathcal{V}$ be the unique solution to \eqref{4.LM} and
let $\bar w\in L^{\infty}(\Omega;\R^N)$. Let $\sigma\in[0,1]$.
We prove that there exists a unique $w\in\widetilde H^2(\Omega;\R^N)$ to
\begin{equation}\label{4.LMw}
  a_2(w,q) = F_2(q)\quad\mbox{for } q\in\widetilde H^2(\Omega;\R^N),
\end{equation}
where for $w$, $q\in \widetilde H^2(\Omega;\R^N)$,
\begin{align*}
  a_2(w,q) &= \eps\int_\Omega(\Delta w\cdot\Delta q+w\cdot q)dz
	+ \int_\Omega \na q:B(\bar w)\na w dz	 \\
	F_2(q) &= -\frac{\sigma}{\tau}\int_\Omega(\rho'(\bar w)-\rho^{k-1})\cdot qdz
  +\sigma\int_\Omega((u\cdot\na)q)\cdot\rho'(\bar w)dz.
\end{align*}
We infer from Lemma \ref{lem.G} (iv) that $a_2(\cdot,\cdot)$ is a bounded
bilinear form on $\widetilde H^2(\Omega;\R^N)$, and from the positive
definiteness of $B(\bar w)$ (see also Lemma \ref{lem.G} {(iv)}) follows that
$$
  a_2(w,w) \ge \eps\int_\Omega(|\Delta w|^2+|w|^2)dz
	\ge C\|w\|_{H^{2}(\Omega)}^2.
$$
Since $\rho'(\bar{w})$ is a bounded function, by Corollary \ref{coro.w2rho},
we infer that $F_2$ is bounded on $\widetilde H^2(\Omega;\R^N)$.
Then the Lax-Milgram lemma provides the existence of a unique solution
$w\in\widetilde H^2(\Omega;\R^N)$ to \eqref{4.LMw}.

{\em Step 3.} This defines the fixed-point mapping $S:L^{\infty}(\Omega;\R^N)\times
[0,1]\to L^{\infty}(\Omega;\R^N)$, $S(\bar w,\sigma)=w$, where $w$ solves \eqref{4.LMw}.
By construction, $S(\bar w,0)=0$ for all $w\in L^{\infty}(\Omega;\R^N)$. Since the
embedding $H^2(\Omega)\hookrightarrow L^{\infty}(\Omega)$ is compact,
standard arguments show that $S$ is continuous and compact. It remains to prove
that there exists a constant $C>0$ such that $\|w\|_{L^{\infty}(\Omega)}\le C$
for all $(w,\sigma)\in L^{\infty}(\Omega;\R^N)\times[0,1]$ satisfying
$w=S(w,\sigma)$.

Let $w\in L^{\infty}(\Omega;\R^N)$ be such a fixed point.
Then it solves \eqref{4.LMw} with $\bar w$ replaced by $w$.
Taking $w\in \widetilde H^2(\Omega;\R^N)$ as a test function,
it follows from Lemma \ref{lem.ent} that
$$
  \frac{\sigma}{\tau}\int_\Omega(\rho'(w)-\rho'(w^{k-1}))\cdot w dz
  + \int_\Omega \na w:B(w)\na w dz + \eps\int_\Omega(|\Delta w|^2 + |w|^2)dz = 0.
$$
By Lemma \ref{lem.hessian}, the entropy density $h$, defined in \eqref{1.h},
is convex. This implies that $h(\rho'(w))-h(\rho^{k-1})\le (dh/d\rho')\cdot
(\rho'(w)-\rho^{k-1}) = w\cdot (\rho'(w)-\rho^{k-1})$
(see Lemma \ref{lem.w}). We infer from the positive definiteness of $B(w)$ (see
Lemma \ref{lem.G} (iv)) that
$$
  \sigma\int_\Omega h(\rho'(w))dz + \eps\tau\int_\Omega(|\Delta w|^2+|w|^2)dz
	\le \sigma\int_\Omega h(\rho^{k-1})dz.
$$
This yields the desired uniform $H^{2}$ bound and hence uniform $L^{\infty}$
bound for $w$.
By the Leray-Schauder fixed-point theorem, there exists a solution
$w\in \widetilde H^2(\Omega;\R^N)$ to \eqref{4.weak.w}.
According to Corollary \ref{coro.w2rho},
we can define $\rho_1(w),\ldots,\rho_N(w)>0$ satisfying $\rho_{N+1}(w):=
1-\sum_{i=1}\rho_i(w)>0$, and we set
$\eta^k=\min_{1\le i\le N+1}\mbox{\rm ess\,inf}_\Omega\rho_i(w)>0$.
Then, by construction, $\rho'(w)\in Y_{\eta^k}$.
\end{proof}


\subsection{Uniform estimates}

Let $u^0\in\mathcal{H}$ and $\rho^0=(\rho^0_1,\ldots,\rho^0_{N+1})$ satisfing
$\rho^0_i\ge 0$ for $i=1,\ldots,N+1$ and $\sum_{i=1}^{N+1}\rho_i^0=1$. We regularize
$u^0$ by $u^0_\tau$ which is the weak solution to $u_\tau^0-\tau\Delta u_\tau^0=u^0$
in $\Omega$ and $u_\tau^0=0$ on $\pa\Omega$. Then it holds that
$$
  \|u_\tau^0\|_{L^2(\Omega)}^2 + 2\tau\|\na u_\tau^0\|_{L^2(\Omega)}^2
	\le \|u^0\|_{L^2(\Omega)}^2
$$
and (at least for a subsequence)
$u_\tau^0\rightharpoonup u^0$ weakly in $L^2(\Omega;\R^d)$.
Let $0<\eta^0\le 1/(2(N+1))$ and define
$$
  \rho_i^{\eta^0} = \frac{\rho_i^0 + 2\eta^0}{1+2\eta^0(N+1)}, \quad i=1,\ldots,N+1.
$$
Then $\rho_i^{\eta^0}\ge\eta^0$ for all $i=1,\ldots,N+1$ and
$\sum_{i=1}^{N+1}\rho_i^{\eta^0}=1$. Finally, let $w^0\in L^{\infty}(\Omega;\R^N)$
be defined by \eqref{3.w}. Applying Lemma \ref{lem.approx} iteratively,
we obtain a sequence of approximate solutions $(u^k,w^k)\in\mathcal{V}\times
\widetilde H^2(\Omega;\R^N)$ to \eqref{4.weak.u}-\eqref{4.weak.rho} such that
$\rho'(w^k)\in Y_{\eta^k}$, where $\eta^k\in(0,1)$.
For the following, we set $\rho^k=\rho'(w^k)$ for $k\ge 0$,
slightly abusing our notation.

\begin{lemma}\label{lem.est12}
For any $1\le k\le M$ and sufficiently small $\eta^0>0$, it holds that
\begin{align}
  \|u^k\|_{L^2(\Omega)}^2 + \sum_{j=1}^k\|u^j-u^{j-1}\|_{L^2(\Omega)}^2
	+ \tau\sum_{j=1}^k\|\na u^j\|_{L^2(\Omega)}^2
	\le \|u^0\|_{L^2(\Omega)}^2 + \|f\|_{L^2(0,T;\mathcal{V}')}^2, \label{4.est1} \\
	\int_\Omega h(\rho^k)dz + C_B\tau\sum_{j=1}^k\|\na\sqrt{x(\rho^j)}\|_{L^2(\Omega)}^2
	+ \eps\tau\sum_{j=1}^k\int_\Omega(|\Delta w^j|^2+|w^j|^2)dz
	\le \int_\Omega h(\rho^0)dz + 1, \label{4.est2}
\end{align}
where $\sqrt{x(\rho^j)}=(\sqrt{x_1(\rho^j)},\ldots,\sqrt{x_{N+1}(\rho^j)})$,
$x_i(\rho^j)=\rho_i^j/(cM_i)$ for $i=1,\ldots,N+1$,
$c=\sum_{k=1}^{N+1}\rho_k^j/M_k$, and $C_B>0$ is obtained from Lemma \ref{lem.sqrtx}.
\end{lemma}

\begin{proof}
The proof of \eqref{4.est1} is standard and we refer to \cite[Section III.4.3]{Tem84}
for a proof. Lemma \ref{lem.sqrtx} and Step 3 of the proof of Lemma \ref{lem.approx}
imply after summation over $j=1,\ldots,k$ that
$$
  \int_\Omega h(\rho^k)dz + C_B\tau\sum_{j=1}^k\|\sqrt{x(\rho^j)}\|_{L^2(\Omega)}^2
	+ \eps\tau\sum_{j=1}^k\int_\Omega(|\Delta w^j|^2+|w^j|^2)dz
	\le \int_\Omega h(\rho^{\eta^0})dz.
$$
By dominated convergence,
$$
  \lim_{\eta^0\to 0}\int_\Omega h(\rho^{\eta^0})dz = \int_\Omega h(\rho^0)dz,
$$
and hence, for sufficiently small $\eta^0>0$,
$$
  \int_\Omega h(\rho^{\eta^0})dz \le \int_\Omega h(\rho^0)dz + 1.
$$
This proves \eqref{4.est2}.
\end{proof}

\begin{lemma}\label{lem.est34}
It holds that
\begin{align}
  \tau\sum_{k=1}^M\left\|\frac{u^k-u^{k-1}}{\tau}\right\|_{\mathcal{V}_2'}^2
	&\le C(u^0,f), \label{4.est3} \\
	\tau\sum_{k=1}^M\|\na x(\rho^k)\|_{L^2(\Omega)}^2
	+ \tau\sum_{k=1}^M\|\na \rho^k\|_{L^2(\Omega)}^2
	+ \tau\sum_{k=1}^M\left\|\frac{\rho^k-\rho^{k-1}}{\tau}
	\right\|_{\widetilde H^2(\Omega)'}^2
	&\le C(u^0,\rho^0,f), \label{4.est4}
\end{align}
where ${\mathcal V}_2'$ is the dual space of ${\mathcal V}_2$,
defined in \eqref{1.space}.
\end{lemma}

\begin{proof}
Again, estimate \eqref{4.est3} is standard; see \cite[Section III.4.3]{Tem84}.
Since $x_i(\rho^k)=\rho^k_i/(cM_i)$ with $c=\sum_{k=1}^{N+1}\rho_k/M_k$
is bounded by one, we find that
$$
  \|\na x(\rho^k)\|_{L^2(\Omega)}
	\le 2\|\sqrt{x(\rho^k)}\|_{L^2(\Omega)}\|\na\sqrt{x(\rho^k)}\|_{L^2(\Omega)}
	\le 2\|\na\sqrt{x(\rho^k)}\|_{L^2(\Omega)}.
$$
Thus, by \eqref{4.est2},
$$
  \tau\sum_{k=1}^M\|\na x(\rho^k)\|_{L^2(\Omega)}^2 \le C(\rho^0).
$$
Then it follows from Lemma \ref{lem.G} (iii) that 
$|\nabla {\rho}^k|\le C|\nabla {x}'(\rho^k)|$. Hence,
$$
  \tau\sum_{k=1}^M\|\na\rho^k\|_{L^2(\Omega)}^2
	\le C\tau\sum_{k=1}^M\|\na x'(\rho^k)\|_{L^2(\Omega)}^2 \le C(\rho^0).
$$
We deduce from \eqref{4.weak.rho}, the boundedness of the elements of $A_0^{-1}(\rho^k)$
(see Lemma \ref{lem.A0}), and the uniform estimate for $u_k$ in $L^2$
(see \eqref{4.est1}) that for $q\in \widetilde H^2(\Omega;\R^N)$,
\begin{align*}
  &\left|\frac{1}{\tau}\int_\Omega(\rho^k-\rho^{k-1})\cdot q dz\right|\\
	&\le \|A_0^{-1}(\rho^k)\|_{L^\infty(\Omega)}\|\na x'(\rho^k)\|_{L^2(\Omega)}
	\|q\|_{L^2(\Omega)} \\
	&\phantom{xx}{}+ \|u^k\|_{L^2(\Omega)}\|\na \rho^k\|_{L^2(\Omega)}
	\|q\|_{L^\infty(\Omega)}
	+ \eps(\|\Delta w^k\|_{L^2(\Omega)}+\|w^k\|_{L^2(\Omega)})\|q\|_{L^2(\Omega)} \\
	&\le C(u^0,f)\big(\|\na x'(\rho^k)\|_{L^2(\Omega)}
	+ \|\na \rho^k\|_{L^2(\Omega)} + \eps\|w^k\|_{H^2(\Omega)}\big)\|q\|_{H^2(\Omega)}.
\end{align*}
Taking into account the above uniform estimates for $\na x'(\rho^k)$ and
$\na\rho^k$ in $L^2$ and the estimate \eqref{4.est2} for $\sqrt{\eps}w^k$
in $H^{2}$, it follows that
\begin{align*}
  \tau\sum_{k=1}^M\left\|\frac{\rho^k-\rho^{k-1}}{\tau}
	\right\|_{\widetilde H^2(\Omega)'}^2
	&\le C(u^0,f)\tau\sum_{k=1}^M\left(\|\na x'(\rho^k)\|_{L^2(\Omega)}^2
	+ \|\na\rho^k\|_{L^2(\Omega)}^2 + \eps^2\|w^k\|^2_{H^2(\Omega)}\right) \\
	&\le C(u^0,\rho^0,f).
\end{align*}
This ends the proof.
\end{proof}


\subsection{Proof of Theorem \ref{thm.ex}}

Define the piecewise constant function
$u^{(\tau)}(x,t) = u^k(x)$, its time shift $(\pi_\tau u^{(\tau)})(x,t)=u^{k-1}(x)$
and the difference quotient
$$
  \pa_t^\tau u^{(\tau)}(x,t) = \frac{u^k(x)-u^{k-1}(x)}{\tau}
$$
for $x\in\Omega$, $(k-1)\tau<t\le k\tau$, $k=1,\ldots,M$.
Similarly, we define $f^{(\tau)}$, $w^{(\tau)}$, $\rho^{(\tau)}$, and
$\pa_t^\tau\rho^{(\tau)}$. Lemmas \ref{lem.est12} and \ref{lem.est34} imply
immediately the following uniform estimates:
\begin{align}
  \|u^{(\tau)}\|_{L^\infty(0,T;L^2(\Omega))}
	+ \|u^{(\tau)}\|_{L^2(0,T;H^1(\Omega))}
	+ \|\pa_t^\tau u^{(\tau)}\|_{L^2(0,T;\mathcal{V}_2')}
	&\le C, \label{4.est.u1} \\
	\tau^{-1}\|\pi_\tau u^{(\tau)}-u^{(\tau)}\|_{L^2(0,T;L^2(\Omega))}
	&\le C, \label{4.est.u2} \\
	\|x'(\rho^{(\tau)})\|_{L^\infty(0,T;L^\infty(\Omega))}
	+ \|x'(\rho^{(\tau)})\|_{L^2(0,T;H^1(\Omega))}
	&\le C, \label{4.est.x} \\
	\|\rho^{(\tau)}\|_{L^\infty(0,T;L^\infty(\Omega))}
	+ \|\rho^{(\tau)}\|_{L^2(0,T;H^1(\Omega))}
	+ \|\pa_t^\tau\rho^{(\tau)}\|_{L^2(0,T;\widetilde H^2(\Omega)')}
	&\le C, \label{4.est.rho} \\
	\sqrt{\eps}\|w^{(\tau)}\|_{L^2(0,T;H^{2}(\Omega))} &\le C. \label{4.est.w}
\end{align}

The weak formulation \eqref{4.weak.u}-\eqref{4.weak.rho} can be written
for any $v\in C_0^\infty(\Omega\times[0,T);\R^d)$ with $\diver v=0$ and any
$q\in C_{0}^\infty(\overline{\Omega}\times[0,T);\R^N)$ with
$\na q\cdot\nu|_{\pa\Omega}=0$ as follows:
\begin{align}
  \int_0^T\int_\Omega &\pa_t^\tau u^{(\tau)}\cdot v dz\,dt
	+ \int_0^T\int_\Omega\na u^{(\tau)}:\na v dz\,dt
	+ \int_0^T\int_\Omega((\pi_\tau u^{(\tau)}\cdot\na)u^{(\tau)})\cdot vdz\,dt
	\nonumber \\
	&{}= \int_0^T\int_\Omega f^{(\tau)}\cdot v dz\,dt, \label{4.eq1} \\
	\int_0^T\int_\Omega &\pa_t^\tau\rho^{(\tau)}\cdot q dz\,dt
	+ \int_0^T\int_\Omega \na q:A_0^{-1}(\rho^{(\tau)})\na x'(\rho^{(\tau)}) dz\,dt
	+ \int_0^T\int_\Omega((u^{(\tau)}\cdot\na)\rho^{(\tau)})\cdot q dz\,dt \nonumber \\
	&{}= -\eps\int_0^T\int_\Omega(\Delta w^{(\tau)}\cdot\Delta q + w^{(\tau)}
	\cdot q)dz\,dt. \label{4.eq2}
\end{align}
Estimates \eqref{4.est.u1} for $(u^{(\tau)})$
and \eqref{4.est.rho} for $(\rho^{(\tau)})$ allow us to apply Aubin's lemma
in the version of \cite{DrJu12} which yields the existence of subsequences
of $(u^{(\tau)})$ and $(\rho^{(\tau)})$ (not relabeled) such that, as
$(\eps,\tau)\to 0$,
$$
  u^{(\tau)}\to u, \quad \rho^{(\tau)}\to\rho'\quad\mbox{strongly in }
	L^2(0,T;L^2(\Omega)).
$$
Consequently, by \eqref{4.est.u2},
$$
  \|\pi_\tau u^{(\tau)}-u\|_{L^2(0,T;L^2(\Omega))}
	\le \|\pi_\tau u^{(\tau)}-u^{(\tau)}\|_{L^2(0,T;L^2(\Omega))}
	+ \|u^{(\tau)}-u\|_{L^2(0,T;L^2(\Omega))}\to 0
$$
and $(\pi_\tau u^{(\tau)}\cdot\na)u^{(\tau)}\rightharpoonup (u\cdot\na)u$
weakly in $L^1(0,T;L^1(\Omega))$.
Furthermore, the strong convergence of $(\rho^{(\tau)})$ and the boundedness
of the elements of $A_0^{-1}$ and $x'$ yield
$A_0^{-1}(\rho^{(\tau)})\to A_0^{-1}(\rho')$, $x'(\rho^{(\tau)})\to x'(\rho')$
strongly in $L^p(0,T;L^p(\Omega))$ for any $p<\infty$.
Together with the weak convergence (again up to a subsequence) of
$(\na x'(\rho^{(\tau)}))$, we infer that
$$
  \na x'(\rho^{(\tau)})\rightharpoonup \na x'(\rho)\quad
	\mbox{weakly in }L^2(0,T;L^2(\Omega)).
$$
Finally, we note that $f^{(\tau)}\to f$ strongly in $L^2(0,T;\mathcal{V}')$
(see \cite[Lemma III.4.9]{Tem84}) and $\eps w^{(\tau)}\to 0$ strongly
in $L^2(0,T;H^{2}(\Omega))$ as $(\eps,\tau)\to 0$. These convergences are sufficient
to pass to the limit $(\eps,\tau)\to 0$ in \eqref{4.eq1}-\eqref{4.eq2}
yielding a global solution $(u,\rho')$ to \eqref{4.weak1}-\eqref{4.weak2}.
In view of the a priori estimates uniform in $\eta^0$ and the finiteness of
the initial entropy, we can perform the limit $\eta^0\to 0$ and hence
conclude the existence result for general initial data. The theorem is proved.


\section{Proof of Theorem \ref{thm.long}}\label{sec.long}

Let $(u^k,w^k)$ be a solution to \eqref{4.weak.u} and \eqref{4.weak.w}.
First, we prove $L^1$ bounds for $\rho_i^k=\rho_i(w^k)$
and $c^k=\sum_{i=1}^{N+1}\rho_i^k/M_k$.

\begin{lemma}[Uniform $L^1$ norms for $\rho^k$]\label{lem.L1rho}
There exist constants $\gamma_0>0$, depending on $\rho^0$, and $\eps_0>0$ such that
for all $0<\gamma<\min\{1,\gamma_0\}$ and $0<\eps<\eps_0$,
\begin{align}
  \big|\|\rho_i^k\|_{L^1(\Omega)}-\|\rho_i^0\|_{L^1(\Omega)}\big|
	&\le \gamma\|\rho_i^0\|_{L^1(\Omega)}, \quad i=1,\ldots,N, \label{5.L11} \\
	\big|\|\rho_{N+1}^k\|_{L^1(\Omega)}-\|\rho_{N+1}^0\|_{L^1(\Omega)}\big|
	&\le \gamma\sum_{I=1}^N\|\rho_i^0\|_{L^1(\Omega)}. \label{5.L12}
\end{align}
Furthermore, $\|\rho_{N+1}^k\|_{L^1(\Omega)}
\ge \frac12\|\rho_{N+1}^0\|_{L^1(\Omega)} > 0$.
\end{lemma}

\begin{proof}
The proof is similar to the proof of Lemma 4.1 in \cite{JuSt13}.
The main difference is that the entropy differs from that of \cite{JuSt13}
which makes some changes necessary. We recall that
$\tau=T/M$ with $T>0$ and $M\in\N$. Using the test function $q=e_i$
in \eqref{4.weak.w}, where $e_i$ is the $i$th unit vector of $\R^N$, and
observing that
$$
  \int_\Omega ((u^k\cdot\na)\rho'(w^k))\cdot e_i dz
	= -\int_\Omega\diver(u^{k})\rho_i(w^k)dz = 0,
$$
we have
$$
  \int_\Omega\rho_i^k dz = \int_\Omega\rho_i^{k-1}
	- \eps\tau\int_\Omega w_i^k dz, \quad i=1,\ldots,N.
$$
Solving this recursion, we deduce that
\begin{equation}\label{5.aux0}
  \int_\Omega\rho_i^k dz = \int_\Omega\rho_i^0 dz
	- \eps\tau\sum_{j=1}^k\int_\Omega w_i^j dz, \quad i=1,\ldots,N.
\end{equation}
Thus, we need to bound the $L^1$ norm of $w_i^j$.
Recalling that $H(\rho^k)=\int_\Omega h(\rho'(w^k))dz$, we infer
from Step 3 of the proof of Lemma \ref{lem.approx} that
$$
  H(\rho^k) + \eps\tau\int_\Omega |w_i^k|^2 dz \le H(\rho^{k-1})
$$
or, solving the recursion,
\begin{equation}\label{5.aux1}
  H(\rho^k) + \eps\tau\sum_{j=1}^k\int_\Omega |w_i^j|^2 dz \le H(\rho^0).
\end{equation}
It follows from the definition of the entropy and estimate \eqref{3.cmax} that
the entropy can be bounded from below:
\begin{align*}
  H(\rho^k)
	&= \int_\Omega c^k\sum_{j=1}^{N+1}\big(x_i^k(\ln x_i^k-1)+1\big)
	- N\int_\Omega c^k dz \ge -C_1:=-N\mbox{meas}(\Omega)M_*^{-1},
\end{align*}
where $c^k=\sum_{i=1}^{N+1}\rho^k_i/M_i$, $x_i^k=\rho_i^k/(c^kM_i)$,
and $M_*=\min_{1\le i\le N+1}M_i$. Therefore, \eqref{5.aux1} implies that
$$
  \eps\tau\sum_{j=1}^k\int_\Omega |w_i^j|^2 dz
	\le H(\rho^0)-H(\rho^k) \le H(\rho^0)+C_1.
$$
The $L^1$ norm of $w_i^k$ can be estimated by its $L^2$ norm
by applying the Cauchy-Schwarz inequality:
\begin{align*}
  \eps\tau\sum_{j=1}^k\int_\Omega|w_i^j|dz
	&\le \eps\tau\sqrt{\mbox{meas}(\Omega)}\sum_{j=1}^k\|w_i^j\|_{L^2(\Omega)}
	\le \eps\tau\sqrt{k\mbox{meas}(\Omega)}\left(\sum_{j=1}^k\|w_i^j\|_{L^2(\Omega)}^2
	\right)^{1/2} \\
	&= \sqrt{\eps\tau k\mbox{meas}(\Omega)}\left(\eps\tau\sum_{j=1}^k
	\|w_i^j\|_{L^2(\Omega)}^2\right)^{1/2}
	\le \sqrt{\eps T\mbox{meas}(\Omega)(H(\rho^0)+C_1)},
\end{align*}
where we used $\tau k\le T$. We conclude from \eqref{5.aux0} that
$$
  \big|\|\rho_i^k\|_{L^1(\Omega)}-\|\rho_i^0\|_{L^1(\Omega)}\big|
	\le \sqrt{\eps T\mbox{meas}(\Omega)(H(\rho^0)+C_1)}.
$$
Given $0<\gamma<1$, let $\eps>0$ satisfy
\begin{equation}\label{5.eps}
  \sqrt{\eps} \le
	\frac{\gamma\min_{1\le j\le N}\|\rho_j^0\|_{L^1(\Omega)}}{
	\sqrt{T\mbox{meas}(\Omega)(H(\rho^0)+C_1)}}.
\end{equation}
This proves \eqref{5.L11}.

For $i=N+1$, we estimate
\begin{align*}
  \big|\|\rho_{N+1}^k\|_{L^1(\Omega)}-\|\rho_{N+1}^0\|_{L^1(\Omega)}\big|
	&= \left|\int_\Omega\left(1-\sum_{i=1}^N\rho_i^k\right)dz
	- \int_\Omega\left(1-\sum_{i=1}^N\rho_i^0\right)dz\right| \\
	&\le \sum_{i=1}^N\big|\|\rho_i^k\|_{L^1(\Omega)}-\|\rho_i^0\|_{L^1(\Omega)}\big|
	\le \gamma\sum_{i=1}^N\|\rho_i^0\|_{L^1(\Omega)}
\end{align*}
which proves \eqref{5.L12}. From this estimate follows that
$$
  \|\rho_{N+1}^k\|_{L^1(\Omega)}
	\ge \|\rho_{N+1}^0\|_{L^1(\Omega)} - \gamma\sum_{i=1}^N\|\rho_i^0\|_{L^1(\Omega)}.
$$
Hence, defining
\begin{equation}\label{5.gamma0}
  \gamma_0 = \frac{\|\rho_{N+1}^0\|_{L^1(\Omega)}}{2
	\sum_{i=1}^N\|\rho_i^0\|_{L^1(\Omega)}}
\end{equation}
and choosing $0<\gamma<\min\{1,\gamma_0\}$,
we deduce that $\|\rho_{N+1}^k\|_{L^1(\Omega)}\ge
\frac12\|\rho_{N+1}^0\|_{L^1(\Omega)}$.
\end{proof}

\begin{lemma}[Uniform $L^1$ norms for $c^k$]\label{lem.L1w}
With $\gamma$ as in Lemma \ref{lem.L1rho}, it holds that
$$
  \big|\|c^k\|_{L^1(\Omega)}-\|c^0\|_{L^1(\Omega)}\big|
	\le M_0\gamma\|c^0\|_{L^1(\Omega)},
$$
where $M_0=\max_{1\le i\le N}|1-M_i/M_{N+1}|$.
\end{lemma}

\begin{proof}
We employ the definitions $c^k=\sum_{i=1}^{N+1}\rho_i^k/M_i$ and
$\sum_{i=1}^{N+1}\rho_i^k=1$ and the estimate \eqref{5.L11} to obtain
\begin{align*}
  \big|\|c^k&\|_{L^1(\Omega)}-\|c^0\|_{L^1(\Omega)}\big|
	= \left|\sum_{i=1}^{N+1}\frac{1}{M_i}\int_\Omega(\rho_i^k-\rho_i^0)dz\right| \\
	&= \left|\sum_{i=1}^{N}\left(\frac{1}{M_i}-\frac{1}{M_{N+1}}\right)
	\int_\Omega(\rho_i^k-\rho_i^0)dz\right|
	\le M_0\sum_{i=1}^N\frac{1}{M_i}
	\big|\|\rho_i^k\|_{L^1(\Omega)}-\|\rho_i^0\|_{L^1(\Omega)}\big| \\
	&\le M_0\gamma\sum_{i=1}^N\frac{\|\rho_i^0\|_{L^1(\Omega)}}{M_i}
	\le M_0\gamma\|c^0\|_{L^1(\Omega)}.
\end{align*}
which finishes the proof.
\end{proof}

Now, we turn to the proof of Theorem \ref{thm.long} which is divided
into several steps.

{\em Step 1: Relative entropy dissipation inequality.}
Let $(u^k,w^k)\in\mathcal{V}\times \widetilde H^2(\Omega;\R^N)$ be a solution
to \eqref{4.weak.u} and \eqref{4.weak.w} which exists according to
Lemma \ref{lem.approx}. We introduce the following notation:
\begin{align*}
  & \rho^k = (\rho_1^k,\ldots,\rho_{N+1}^k) = (\rho_1(w^k),\ldots,\rho_{N+1}(w^k)),
	\quad w^k = (w_1^k,\ldots,w_N^k), \\
	& \bar\rho^k = (\bar\rho_1^k,\ldots,\bar\rho_{N+1}^k), \quad
	\bar x^k = (\bar x_1^k,\ldots,\bar x_{N+1}^k), \quad
	\bar w^k = (\bar w_1^k,\ldots,\bar w_{N}^k),
\end{align*}
where $\bar\rho_i^k = \mbox{meas}(\Omega)^{-1}\|\rho_i^k\|_{L^1(\Omega)}$,
$\bar c^k =\mbox{meas}(\Omega)^{-1}\|c^k\|_{L^1(\Omega)}$,
$\bar x_i^k=\bar\rho_i^k/(\bar c^k M_i)$ for $i=1,\ldots,N+1$, and
$\bar w_i^k=\ln(\bar x_i^k)/M_i-\ln(\bar x_{N+1}^k)/M_{N+1}$ for $i=1,\ldots,N$.
It holds that
$$
  \bar c^k = \sum_{i=1}^{N+1}\frac{\bar\rho_i^k}{M_i}, \quad
	\sum_{i=1}^{N+1}\bar\rho_i^k = \sum_{i=1}^{N+1}\bar x_i^k = 1.
$$
With the test function $w^k-\bar w^k$ in \eqref{4.weak.w} we obtain
\begin{align}
  \frac{1}{\tau}& \int_\Omega(\rho'(w^k)-\rho'(w^{k-1}))\cdot(w^k-\bar w^k)dz
	+ \int_\Omega \na w^k:B(w^k)\na w^k dz \nonumber \\
	&{}+ \int_\Omega((u^k\cdot\na)\rho'(w^k))\cdot(w^k-\bar w^k)dz
	+ \eps\int_\Omega(|\Delta w^k|^2+ w^k\cdot(w^k-\bar w^k))dz = 0. \label{5.weak}
\end{align}
If $k=1$, we write $(\rho_1,\ldots,\rho_N)$ instead of $\rho'(w^{k-1})$
in the first integral.
The second integral can be estimated according to Lemma \ref{lem.sqrtx} and the
third integral vanishes in view of Lemma \ref{lem.ent}. Furthermore, using
$w^k\cdot(w^k-\bar w^k) \ge \frac12(|w^k|^2-|\bar w^k|)^2$,
the fourth integral can be written as
$$
  \int_\Omega(|\Delta w^k|^2+ w^k\cdot(w^k-\bar w^k))dz
	\ge \frac12\int_\Omega(|w^k|^2-|\bar w^k|^2)dz \ge -\frac12\int_\Omega|\bar w^k|^2 dz.
$$

It remains to treat the first integral in \eqref{5.weak}.
For this, we employ the formulation \eqref{3.w}
of $w^k$ and $\rho_{N+1}^k=1-\sum_{i=1}^N\rho_i^k$:
\begin{align*}
  (\rho'(w^k)-\rho'(w^{k-1}))\cdot w^k
	&= \sum_{i=1}^N(\rho_i^k-\rho_i^{k-1})\left(\frac{\ln x_i^k}{M_i}
	-\frac{\ln x_{N+1}^k}{M_{N+1}}\right)
	= \sum_{i=1}^{N+1}(\rho_i^k-\rho_i^{k-1})\frac{\ln x_i^k}{M_i} \\
	&= \sum_{i=1}^{N+1}(c^k x_i^k-c^{k-1}x_i^{k-1})\ln x_i^k
	= (c^kx^k-c^{k-1}x^{k-1})\cdot\ln x^k.
\end{align*}
Similarly, $(\rho'(w^k)-\rho'(w^{k-1}))\cdot\bar w^k
=(c^kx^k-c^{k-1}x^{k-1})\cdot\ln\bar x^k$. Therefore, the first integral becomes
\begin{align*}
  \int_\Omega(&\rho'(w^k)-\rho'(w^{k-1}))\cdot(w^k-\bar w^k)dz
	= \int_\Omega(c^kx^k-c^{k-1}x^{k-1})\cdot\ln\frac{x^k}{\bar x^k}dz \\
	&=\int_\Omega(c^kx^k-c^{k-1}x^{k-1})\cdot\ln\frac{x^k}{\bar x^0}dz
	+ \int_\Omega (c^kx^k-c^{k-1}x^{k-1})\cdot\ln\frac{\bar x^0}{\bar x^k}dz
	= I_1+I_2.
\end{align*}
First, we estimate $I_1$ . To this end, we use the convexity of $h(\rho')$:
\begin{equation}\label{5.aux2}
  h(\rho'(w^k))-h(\rho'(w^{k-1})) \le w^k\cdot(\rho'(w^k)-\rho'(w^{k-1}))
	= (c^kx^k-c^{k-1}x^{k-1})\cdot\ln x^k.
\end{equation}
Then definitions \eqref{1.relent} of the relative entropy $H^*$ and
\eqref{1.h} of the entropy density $h(\rho')$ give
\begin{align*}
  H^*(\rho^k)-H^*(\rho^{k-1})
	&= \sum_{i=1}^{N+1}\int_\Omega(c^kx_i^k\ln x_i^k-c^{k-1}x_i^{k-1}\ln x_i^{k-1})dz \\
	&\phantom{xx}{}- \int_\Omega(c^kx^k-c^{k-1}x^{k-1})\cdot\ln \bar x^0 dz \\
	&= \int_\Omega\big(h(\rho'(w^k))-h(\rho'(w^{k-1}))\big)dz
	+ \sum_{i=1}^{N+1}\int_\Omega(c^kx_i^k-c^{k-1}x_i^{k-1})dz \\
	&\phantom{xx}{}- \int_\Omega(c^k-c^{k-1})dz
	- \int_\Omega(c^kx^k-c^{k-1}x^{k-1})\cdot\ln \bar x^0 dz.
\end{align*}
Since $\sum_{i=1}^{N+1}\int_\Omega c^k x_i^k dz = \int_\Omega c^k dz$,
the second and third integrals on the right-hand side cancel.
We employ \eqref{5.aux2} to find that
$$
  H^*(\rho^k)-H^*(\rho^{k-1}) \le \int_\Omega(c^kx^k-c^{k-1}x^{k-1})\cdot\ln x^k dz
	- \int_\Omega(c^kx^k-c^{k-1}x^{k-1})\cdot\ln \bar x_i^0 dz = I_1.
$$

Next, we estimate $I_2$. Let $0<\gamma<\min\{\frac12,\gamma_0,(2M_0)^{-1}\}$,
where $M_0$ is defined in Lemma \ref{lem.L1w}. We infer from
Lemmas \ref{lem.L1rho} and \ref{lem.L1w} and from the definition \eqref{5.gamma0}
of $\gamma_0$ the following bounds:
\begin{align}
  & \frac{1-M_0\gamma}{1+\gamma} \le \frac{\bar x_i^0}{\bar x_i^k}
	= \frac{\|\rho_i^0\|_{L^1(\Omega)}\|c^k\|_{L^1(\Omega)}}{\|\rho_i^k\|_{L^1(\Omega)}
	\|c^0\|_{L^1(\Omega)}}
	\le \frac{1+M_0\gamma}{1-\gamma}, \quad i=1,\ldots,N, \label{5.x1} \\
	& \frac{1-M_0\gamma}{1+\gamma/(2\gamma_0)} \le \frac{\bar x_{N+1}^0}{\bar x_{N+1}^k}
	= \frac{\|\rho_{N+1}^0\|_{L^1(\Omega)}\|c^k\|_{L^1(\Omega)}}{
	\|\rho_{N+1}^k\|_{L^1(\Omega)}\|c^0\|_{L^1(\Omega)}}
	\le \frac{1+M_0\gamma}{1-\gamma/(2\gamma_0)}. \label{5.x2}
\end{align}
Thus, taking into account $\sum_{i=1}^{N+1}x_i^k=1$, we obtain
\begin{align*}
  I_2 &\ge \sum_{i=1}^N\int_\Omega c^kx_i^k dz\ln\frac{1-M_0\gamma}{1+\gamma}
	+ \int_\Omega c^kx_{N+1}^k dz\ln\frac{1-M_0\gamma}{1+\gamma/(2\gamma_0)} \\
	&\phantom{xx}{}- \sum_{i=1}^N\int_\Omega c^{k-1}x_i^{k-1} dz\
	\ln\frac{1+M_0\gamma}{1-\gamma}
	- \int_\Omega c^{k-1}x_{N+1}^{k-1}dz\ln\frac{1+M_0\gamma}{1-\gamma/(2\gamma_0)} \\
	&\ge \int_\Omega c^k dz\ln\frac{1-M_0\gamma}{(1+\gamma)(1+\gamma/(2\gamma_0))}
	- \int_\Omega c^{k-1}dz\ln\frac{1+M_0\gamma}{(1-\gamma)(1-\gamma/(2\gamma_0))}.
\end{align*}
Because of $c^k\le (\min_{1\le i\le N+1}M_i)^{-1}=M_*^{-1}$ (see \eqref{3.cmax}),
we conclude that
\begin{equation}\label{5.C2}
  I_2 \ge -C_2(\gamma) := -\mbox{meas}(\Omega)M_*^{-1}\ln\frac{(1+M_0\gamma)(1+\gamma)
	(1+\gamma/(2\gamma_0))}{(1-M_0\gamma)(1-\gamma)	(1-\gamma/(2\gamma_0))}.
\end{equation}
Therefore, the first integral in \eqref{5.weak} is bounded as follows:
$$
  \int_\Omega(\rho'(w^k)-\rho'(w^{k-1}))\cdot(w^k-\bar w^k)dz
  \ge H^*(\rho^k)-H^*(\rho^{k-1}) - C_2(\gamma).
$$
Summarizing, \eqref{5.weak} can be estimated as
\begin{equation}\label{5.disc}
  H^*(\rho^k) - H^*(\rho^{k-1}) + C_B\tau\int_\Omega\|\na\sqrt{x^k}\|^2 dz
	\le \frac{\eps\tau}{2}\int_\Omega|\bar w^k|^2 dz + C_2(\gamma).
\end{equation}

{\em Step 2: Estimate of the relative entropy.}
We split the relative entropy into two integrals:
$$
  H^*(\rho^k) = \sum_{i=1}^{N+1}\int_\Omega c^k x_i^k\ln\frac{x_i^k}{\bar x_i^k}dz
	+ \sum_{i=1}^{N+1}\int_\Omega c^k x_i^k\ln\frac{\bar x_i^k}{\bar x_i^0}dz
	= J_1 + J_2.
$$
It follows from \eqref{5.x1} and \eqref{5.x2} that
\begin{align}
  J_2 &\le \sum_{i=1}^N\int_\Omega c^k x_i^k dz\ln\frac{1+\gamma}{1-M_0\gamma}
	+ \int_\Omega c^k x_{N+1}^k dz\ln\frac{1+\gamma/(2\gamma_0)}{1-M_0\gamma}
	\nonumber \\
	&\le C_3(\gamma):=\mbox{meas}(\Omega)M_*^{-1}
	\ln\frac{(1+\gamma)(1+\gamma/(2\gamma_0)}{1-M_0\gamma}. \label{5.C3}
\end{align}
The integral $J_1$ is also split into two parts:
$$
  J_1 = \sum_{i=1}^{N+1}\int_\Omega c^k x_i^k
	\ln\frac{c^kx_i^k\mbox{meas}(\Omega)}{\|c^k x_i^k\|_{L^1(\Omega)}} dz
	+ \sum_{i=1}^{N+1}\int_\Omega c^k x_i^k
	\ln\frac{\|c^kx_i^k\|_{L^1(\Omega)}}{c^k\bar x_i^k\mbox{meas}(\Omega)} dz
	= J_{11}+J_{12}.
$$
Inserting the definitions $x_i^k=\rho_i^k/(c^kM_i)$ and
$\bar x_i^k=\bar\rho_i^k/(\bar c^k M_i)$ and using Jensen's inequality
for the convex function $s\mapsto s\ln s$ ($s>0$), we obtain
$$
  J_{12} = \sum_{i=1}^{N+1}\int_\Omega c^kx_i^k\ln\frac{\bar c^k}{c^k} dz
	= \int_\Omega c^k\ln\frac{\bar c^k}{c^k}dz
	= \|c^k\|_{L^1(\Omega)}\ln\bar c^k - \|c^k\ln c^k\|_{L^1(\Omega)} \le 0.
$$
The estimate of $J_{11}$ is more involved. We employ the logarithmic Sobolev
inequality
$$
  \int_\Omega u^2\ln\frac{u^2}{\bar u^2}dz \le C_L\int_\Omega|\na u|^2 dz,
	\quad \bar u^2 = \frac{1}{\mbox{meas}(\Omega)}\int_\Omega u^2 dz,
$$
where $u\in H^1(\Omega)$, and $C_L>0$ depends only on $\Omega$ \cite{Jue12}. Then
$$
  J_{11} \le C_L\sum_{i=1}^{N+1}\int_\Omega|\na\sqrt{c^k x_i^k}|^2 dz.
$$
Since
$$
  \sum_{i=1}^{N+1}|\na\sqrt{c^k x_i^k}|^2
	\le 2\sum_{i=1}^{N+1}x_i^k|\na\sqrt{c^k}|^2 + 2\sum_{i=1}^{N+1}c^k|\na\sqrt{x_i^k}|^2
	= 2|\na\sqrt{c^k}|^2 + 2c^k\|\na\sqrt{x^k}\|^2,
$$
we obtain
$$
  J_{11} \le 2C_L\int_\Omega|\na\sqrt{c^k}|^2 dz
	+ 2C_LM_*^{-1}\int_\Omega\|\na\sqrt{x^k}\|^2 dz.
$$
We claim that the first integral can be estimated by a multiple of the second one.
Indeed, by the Cauchy-Schwarz inequality, the definition of $c^k$ according to
Lemma \ref{lem.x2rho}, and the bound \eqref{3.cmax}, it follows that
\begin{align*}
  |\na\sqrt{c^k}|^2
	&= \frac{1}{4c^k}\left|\frac{-\sum_{i=1}^{N+1}M_i\na x_i^k}{(\sum_{i=1}^{N+1}
	M_i x_i^k)^2}\right|^2
	= (c^k)^3\left|\sum_{i=1}^{N+1}M_i\sqrt{x_i^k}\na\sqrt{x_i^k}\right|^2 \\
	&\le (c^k)^3\sum_{i=1}^{N+1}M_i^2 x_i^k\sum_{i=1}^{N+1}|\na\sqrt{x_i^k}|^2
	\le M_*^{-3}M^{*2}\|\na\sqrt{x^k}\|^2,
\end{align*}
recalling that $M_*=\min_{1\le i\le N+1}M_i$ and setting $M^*=\max_{1\le i\le N+1}M_i$.
Thus, we can estimate $J_{11}$ as follows:
$$
  J_{11} \le 2C_LM_*^{-1}(M_*^{-2}M^{*2}+1)\int_\Omega\|\na\sqrt{x^k}\|^2 dz.
$$
Combining the above estimates, we conclude that
$$
  H^*(\rho^k) \le C_3(\gamma)
	+ 2C_LM_*^{-1}(M_*^{-2}M^{*2}+1)\int_\Omega\|\na\sqrt{x^k}\|^2 dz.
$$

{\em Step 3: End of the proof.} Replacing the entropy dissipation term
involving $\sqrt{x^k}$ in \eqref{5.disc} by the above estimate for $H^*(\rho^k)$,
we find that
\begin{equation}\label{5.aux3}
  (1+C_4\tau)H^*(\rho^k) \le H^*(\rho^{k-1})
	+ \frac{\eps\tau}{2}\int_\Omega |\bar w^k|^2 dz + C_\gamma,
\end{equation}
where $C_4=\frac12 C_BC_L^{-1}M_*(M_*^{-2}M^{*2}+1)^{-1}$ and
$C_\gamma=C_2(\gamma)+\frac12C_3(\gamma)C_L^{-1}M_*(M_*^{-2}M^{*2}+1)^{-1}$.
Note that according to definitions \eqref{5.C2}
and \eqref{5.C3}, we have $C_\gamma\to 0$ as $\gamma\to 0$.

We need to estimate the integral involving $w^k$. For this, we observe that
\eqref{5.x1}-\eqref{5.x2} and the upper bound for $\gamma$ imply that
$\frac13\le \bar x_i^0/\bar x_i^k\le 3$ for $i=1,\ldots,N+1$. This provides
some uniform bounds for $\bar x_i^k$,
$$
  0 < \frac{M_*\min_{1\le i\le N+1}\|\rho_i^0\|_{L^1(\Omega)}}{3M^*
	\sum_{i=1}^{N+1}\|\rho_i^0\|_{L^1(\Omega)}} \le \frac{\bar x_i^0}{3}
	\le \bar x_i^k \le 3\bar x_i^0 \le 3, \quad i=1,\ldots,N+1,
$$
which allow us to estimate $w^k$:
$$
  \int_\Omega|\bar w^k|^2 dz \le \sum_{i=1}^{N}\int_\Omega\left(
	\left|\frac{\ln\bar x_i^k}{M_i}\right|
	+ \left|\frac{\ln\bar x_{N+1}^k}{M_{N+1}}\right|\right)^2 dz \le C_5,
$$
where $C_5>0$ depends on $\Omega$, $\rho^0$, $M_*$, and $M^*$. Hence, \eqref{5.aux3}
becomes
$$
  H^*(\rho^k) \le (1+C_4\tau)^{-1}H^*(\rho^{k-1})
	+ \left(\frac{\eps\tau}{2}C_5 + C_\gamma\right)(1+C_4\tau)^{-1}.
$$
Solving this recursion, we infer that
$$
  H^*(\rho^k) \le (1+C_4\tau)^{-1}H^*(\rho^{0})
	+ \left(\frac{\eps\tau}{2}C_5 + C_\gamma\right)\sum_{i=1}^k(1+C_4\tau)^{-i}.
$$
Using $\sum_{i=1}^k(1+C_4\tau)^{-i}\le  1/(C_4\tau)$, it follows that
$$
  H^*(\rho^{(\tau)}(\cdot,t)) \le (1+C_4\tau)^{-t/\tau}H^*(\rho^0)
	+ \frac{\eps C_5}{2C_4} + \frac{C_\gamma}{C_4\tau}, \quad 0<t<T.
$$

Now, we take $\tau=\tau(\gamma)=\sqrt{C_\gamma}$ and $\eps=\eps(\gamma)$ according to
\eqref{5.eps}. In the limit $\gamma\to 0$, it follows that
$C_\gamma/\tau(\gamma)\to 0$, $\eps(\gamma)\to 0$, and $\tau(\gamma)\to 0$
so that $\rho_i^{(\tau)}\to \rho_i$
strongly in $L^2(0,T;L^2(\Omega))$ for $i=1,\ldots,N+1$. This gives
in the limit $\gamma\to 0$
\begin{equation}\label{5.decay}
  H^*(\rho(\cdot,t)) \le e^{-C_4 t}H^*(\rho^0), \quad t\ge 0,
\end{equation}
and, taking into account Lemmas \ref{lem.L1rho} and \ref{lem.L1w},
we conclude the $L^1$ conservation for $\rho_i$ and $c$:
$$
  \int_\Omega\rho_i dz = \int_\Omega\rho_i^0 dz, \quad
	\int_\Omega c dz = \int_\Omega c^0 dz,
$$
where $c^0=\sum_{j=1}^{N+1}\rho_j^0/M_j$ and $i=1,\ldots,N+1$.

It remains to estimate $x_i-\bar x_i^0$ in the $L^1$ norm. Defining
$$
  f_i = \frac{cx_i}{\int_\Omega c^0 x_i^0 dz}, \quad g_i = \frac{c}{\int_\Omega c^0 dz},
$$
the entropy $H^*(\rho)=\sum_{i=1}^{N+1}\int_\Omega cx_i\ln(x_i/\bar x_i^0)dz$ can be
written as
$$
  H^*(\rho) = \sum_{i=1}^{N+1}\int_\Omega c^0x_i^0 dz
	\int_\Omega f_i\ln\frac{f_i}{g_i}dz,
$$
where we employed the identity
$$
  \frac{f_i}{g_i} = \frac{x_i\int_\Omega c^0 dz}{\int_\Omega c^0 x_i^0 dz}
	= \frac{M_ix_i\int_\Omega c^0 dz}{\int_\Omega\rho_i^0 dz}
	= \frac{M_i x_i\bar c^0}{\bar\rho_i^0} = \frac{x_i}{\bar x_i^0}.
$$
Finally, using
$$
  c\bar x_i^0 = \frac{c\bar\rho_i^0}{\bar c^0 M_i}
	= \frac{c\int_\Omega\rho_i^0 dz}{\int_\Omega c^0 dzM_i}
	= \frac{c\int_\Omega c^0 x_i^0 dz}{\int_\Omega c^0 dz}
	= \int_\Omega c^0 x_i^0 dz\, g_i
$$
and the Csisz\'ar-Kullback inequality with constant $C_K>0$
(see, e.g., \cite{Jue12,UAMT00}), we find that
\begin{align*}
  \|cx_i-c\bar x_i^0\|_{L^1(\Omega)}^2
	&= \left(\int_\Omega c^0 x_i^0 dz\right)^2\|f_i-g_i\|_{L^1(\Omega)}^2
	\le \int_\Omega c^0 x_i^0 dz\left(\int_\Omega \frac{\rho_i^0}{M_i} dz\right)
	C_K\int_\Omega f_i\ln\frac{f_i}{g_i}dz \\
	&\le M_i^{-1} C_K\|\rho_i^0\|_{L^1(\Omega)} H^*(\rho).
\end{align*}
Together with \eqref{5.decay}, the conclusion of the theorem follows.


\end{document}